\documentclass[a4paper,fleqn]{cas-sc}

\usepackage[numbers, sort&compress]{natbib}

\usepackage[utf8]{inputenc} 
\usepackage{accents}				
\usepackage{todonotes}
\usepackage{subcaption}
\usepackage{stmaryrd} 
\usepackage{amsthm} 
\usepackage[makeroom]{cancel} 
\usepackage{bookmark} 
\usepackage{empheq}
\usepackage{ulem}
\usepackage{booktabs} 

\newcommand{\change}[1]{{\leavevmode\color{orange}#1}}

\newcommand{\norm}[1]{\left\lVert#1\right\rVert}
\newcommand{\mat}[1]{\underline{\mathbf{#1}}}
\newcommand\acclrvec[1]{\accentset{\,\leftrightarrow}{#1}}	
\newcommand{\blocktensor}[1]{\acclrvec{{\mathbf #1}}}	
	
\newcommand{\Nabla} {\vec{\nabla}}
\newcommand{\numfluxb}[1]{\hat{\mathbf{#1}} }

\newcommand{\bigpartialderiv}[2]{ \frac{\partial {#1}}{\partial {#2} } }

\newcommand\stateG[1]{\boldsymbol #1}			

\newcommand{\noncon}{\stateG{\Upsilon}}	

\newcommand{\entVar}{{\mathbf{v}}}

\newcommand\state[1]{\mathbf{#1}}

\newcommand{\DG}{{\mathrm{DG}}}
\newcommand{\FV}{{\mathrm{FV}}}

\newcommand{\sloc}{{\mathrm{loc}}}
\newcommand{\sjump}{{\mathrm{jump}}}

\newcommand{\Jan}{\stateG{\Phi}}

\newcommand{\numnonconsD}[1]{ #1^{\Diamond} }
\newcommand{\numnonconsS}[1]{ #1^{\star} }

\newcommand{\avg}[1]{\left\{\hspace*{-3pt}\left\{#1\right\}\hspace*{-3pt}\right\}}

\newcommand{\jump}[1]{\ensuremath{\left\llbracket #1 \right\rrbracket}}

\newcommand{\numnonconsDxi}[1]{ #1^{1\Diamond} }
\newcommand{\numnonconsDeta}[1]{#1^{2\Diamond} }

\newcommand{\numnonconsSxi}[1]{ #1^{1\star} }
\newcommand{\numnonconsSeta}[1]{#1^{2\star} }

\def\NN{\mathcal{N}} 

\newtheorem{proposition}{Proposition}
\newtheorem{corollary}{Corollary}
\newtheorem{remark}{Remark}

\begin{document}

\let\WriteBookmarks\relax
\def\floatpagepagefraction{1}
\def\textpagefraction{.001}

\shorttitle{Well-Balanced Subcell Limiting for DG Discretization of the SWE}
\shortauthors{Rueda-Ramírez et al.}

\title [mode = title]{Well-Balanced Subcell Limiting for Discontinuous Galerkin Discretizations of the Shallow-Water Equations}

\author[1]{Andrés M. Rueda-Ramírez}[orcid=0000-0001-6557-9162]
\cormark[1]
\ead{am.rueda@upm.es}
\credit{Conceptualization, Formal analysis, Methodology, Software. Visualization, Writing – original draft}

\author[2]{Patrick Ersing}[orcid=0009-0005-3804-5380]
\credit{Formal analysis, Methodology, Software, Validation, Visualization, Writing – original draft}

\author[2]{Andrew R. Winters}
[orcid=0000-0002-5902-1522]
\credit{Formal analysis, Methodology, Software, Supervision, Visualization, Writing – original draft}

\author[3]{Gregor J. Gassner}
[orcid=0000-0002-1752-1158]
\credit{Funding acquisition, Methodology, Writing – original draft}

\address[1]{ETSIAE-UPM-School of Aeronautics, Universidad Politécnica de Madrid, Madrid, Spain}

\address[2]{Department of Mathematics, Applied Mathematics, Linköping University, 581 83, Linköping, Sweden}
\address[3]{Department of Mathematics and Computer Science, University of Cologne, Weyertal 86-90, 50931 Cologne, Germany}

\cortext[cor1]{Corresponding author}





\begin{abstract}
High-order discontinuous Galerkin (DG) methods equipped with subcell finite-volume (FV) limiters provide an efficient framework for the simulation of nonlinear hyperbolic balance laws featuring shocks and complex flow structures. 
However, for systems with non-conservative terms, the design of hybrid DG/FV schemes that simultaneously guarantee high-order accuracy for smooth solutions, robustness, and well-balancedness remains challenging.
In particular, for the shallow water equations with variable bottom topography, standard flux-differencing formulations combined with node-wise subcell limiting generally destroy the well-balanced property, even if both the underlying DG and FV discretizations are individually well-balanced.

In this work, we propose a novel flux-differencing formulation for non-conservative systems that enables node-wise subcell limiting while preserving steady states exactly. 
The key idea is to construct staggered DG fluxes whose non-conservative contributions are expressed in local-times-jump form and vanish individually at equilibrium. 
To achieve this structure, we introduce a suitable reformulation of the shallow water equations in which the source term is proportional to the gradient of the total water height. 
This reformulation allows the design of staggered fluxes that preserve equilibrium locally at the node level, thereby making arbitrary nodal blending with low-order FV fluxes admissible.

The resulting hybrid DG/FV method is high-order accurate, robust, and exactly well-balanced under node-wise limiting. 
Numerical experiments, including challenging two-dimensional dam-break configurations with wet/dry fronts and complex obstacle interactions, demonstrate the improved stability and accuracy of the proposed approach compared to existing subcell limiting strategies.

Although this work focuses on the shallow water equations, the well-balanced hybrid DG/FV methods developed here are applicable to a broader class of nonlinear systems of balance laws, provided certain requirements are satisfied in the discretization of the non-conservative terms.
\end{abstract}




\maketitle

\section{Introduction}

Hyperbolic systems of balance laws arise in a wide range of applications in fluid dynamics, geophysics, plasma physics, and multiphase flows. 
Many of these models contain non-conservative terms that play a crucial structural role, such as topographic source terms in the shallow water equations, divergence-related terms in magnetohydrodynamics, or gravitational forcing in compressible gas dynamics. 
Accurate and robust discretizations of such systems require numerical schemes that carefully balance conservative flux gradients and non-conservative contributions.

High-order discontinuous Galerkin (DG) methods have become a prominent tool for the approximation of nonlinear hyperbolic PDEs due to their flexibility, high-order accuracy, and favorable stability properties \cite{Wang2013High,Cockburn2000,Hindenlang2012, ranocha2021efficient}. 
In particular, split-form DG methods formulated in flux-differencing form allow the construction of entropy-conservative and entropy-stable schemes by employing symmetric two-point numerical fluxes \cite{Fisher2013a, Carpenter2014, Gassner2013, Gassner2016}. 
For nonlinear balance laws, suitable discretizations of non-conservative terms can be incorporated into this framework \cite{Bohm2018, wintermeyer2017entropy, rueda2023entropy, rueda2025entropy, coquel2021entropy, waruszewski2022entropy}, leading to provably stable and consistent high-order schemes.

Despite these advantages, purely high-order DG methods may exhibit non-physical oscillations in the presence of strong shocks, wet/dry interfaces, or under-resolved flow features. 
A widely adopted remedy is the use of hybrid DG/finite-volume (FV) subcell limiting strategies, in which the high-order fluxes are blended locally with robust low-order FV fluxes \cite{Hennemann2020, Rueda-Ramirez2020, RUEDARAMIREZ2022, rueda2023monolithic,mateo2023flux}. 
Such methods combine the accuracy of DG schemes in smooth regions with the admissibility properties of FV methods near discontinuities.

For conservative systems, the construction of subcell limiting procedures that preserve important structural properties, such as conservation and entropy stability, are well understood \cite{kuzmin2020monolithic,guermond2019invariant,Pazner2020, RUEDARAMIREZ2022,lin2024high}. 
For non-conservative systems of balance laws, however, additional difficulties arise. 
In particular, the preservation of steady-state solutions (well-balancedness) becomes delicate when DG and FV fluxes are blended locally at the subcell level.

The shallow water equations with variable bottom topography provide a canonical example.
These equations are commonly written as
\begin{equation}\label{eq:swe_std}
    \bigpartialderiv{}{t}
    \begin{pmatrix}
        h \\
        h v
    \end{pmatrix}
    +
    \bigpartialderiv{}{x}
    \begin{pmatrix}
        h v \\
        h v^2 + \frac{1}{2} g h^2
    \end{pmatrix}
    +
    g h \, 
    \bigpartialderiv{}{x}
    \begin{pmatrix}
        0 \\
        b
    \end{pmatrix}
    =
    \begin{pmatrix}
        0 \\
        0
    \end{pmatrix},
\end{equation}
where $h$ denotes the fluid depth, $v$ the velocity, $g$ the gravitational acceleration, and $b$ the bottom topography.
This system of balance laws must satisfy the well-known lake-at-rest property, in which the total fluid depth, $h+b$, remains constant with vanishing velocity, $v=0$. 
Numerous well-balanced DG and FV discretizations have been developed that preserve both entropy properties and the lake-at-rest steady state exactly \cite{wintermeyer2017entropy, ranocha2017shallow, winters2025trixi, ersing2025entropy, fjordholm2011well}. 

To the authors' knowledge, the only currently available strategy to combine DG and FV discretizations of non-conservative systems of balance laws at the node level is the approach proposed by \citet{rueda2024flux}. 
That framework requires the semi-discrete operator to be written in flux-differencing form, where the non-conservative contributions are expressed as sums of two-point interactions that must be representable as a product of a quantity evaluated locally at a node and a symmetric two-point term. 
This structure guarantees that arbitrary node-wise blending between high-order DG and low-order FV fluxes preserves consistency.

However, when existing well-balanced discretizations of the shallow water equations (e.g., the ones in \cite{fjordholm2011well,wintermeyer2017entropy}) are rewritten in such a local-times-symmetric form and subsequently combined with node-wise subcell limiting, the well-balanced property is generally lost. 
The difficulty is structural rather than algebraic. 
Although the full semi-discrete residual vanishes for the lake-at-rest equilibrium, this cancellation typically occurs only after summation over all nodal interactions. 
Individual staggered flux contributions do not vanish separately. 
As a consequence, introducing arbitrary nodal blending coefficients alters the delicate cancellation mechanism between different node pairs, thereby destroying the discrete balance between flux gradients and source terms.

For this reason, current well-balanced hybrid DG/FV methods for the shallow water equations are restricted to element-wise limiting strategies \cite{careaga2026entropy, ersing2025entropy}, where the entire high-order operator is replaced locally and the equilibrium-preserving cancellation remains intact. 
Such approaches, however, are significantly less local and therefore reduce the potential adaptivity and resolution advantages of node-wise subcell limiting.

The main objective of this work is to overcome this limitation. 
We design a novel flux-differencing formulation for non-conservative systems that guarantees that each staggered DG flux vanishes individually at equilibrium. 
This property ensures that well-balancedness holds locally at the node level and therefore remains intact under arbitrary nodal blending with low-order FV fluxes.

The key ingredient is a local-times-jump representation of the non-conservative term, in which the numerical contribution is proportional to the jump of a physically relevant equilibrium quantity. 
For the shallow water equations, we introduce a reformulation of the equations that is closely related to the formulation of \citet{ersing2025entropy}, in which the source term depends on the gradient of the total water height. 
Since this quantity is constant in the lake-at-rest steady state, the corresponding jump vanishes identically, and every staggered flux contribution is zero at equilibrium.

Based on this structure, we derive new staggered DG fluxes that (i) recover the original high-order split-form DG discretization, (ii) vanish individually for equilibrium states, and (iii) are compatible with node-wise subcell limiting.

The proposed method is implemented within the high-order DG framework provided by {Trixi.jl} \cite{ranocha2022adaptive, schlottkelakemper2021purely, schlottkelakemper2025trixi} and its shallow water extension {TrixiShallowWater.jl} \cite{winters2025trixi}. 
Its performance is assessed through a set of numerical experiments for the shallow water equations with variable bottom topography. 
These experiments include well-balancedness verification tests, symmetric dam-break configurations, and challenging two-dimensional flows past obstacles involving wet/dry interfaces. 
The results demonstrate that the proposed node-wise limiting strategy preserves steady states exactly, maintains symmetry in symmetric configurations, and enhances robustness compared to previously available hybrid DG/FV approaches.

The remainder of this paper is organized as follows. 
Section~\ref{sec:state-of-art} reviews the state of the art, introducing the high-order split-form DG discretization for non-conservative systems, discussing subcell limiting techniques, and analyzing the structural loss of well-balancedness under node-wise blending. 
Section~\ref{sec:fix-limiting} presents a reformulation of the shallow water equations together with the derivation of the new flux-differencing formulation. 
Numerical results are reported in Section~\ref{sec:Application}, followed by concluding remarks.

\section{State of the art} \label{sec:state-of-art}

For clarity and compactness, in this section, we present the numerical discretization schemes in the setting of a one-dimensional non-conservative system of balance laws. 
The extension to higher-dimensional systems on general curvilinear meshes follows naturally through tensor-product constructions, as explained in Appendix \ref{app:2d}.

Although this work focuses on the shallow water equations, the well-balanced hybrid DG/FV methods developed here are applicable to a broader class of nonlinear systems of balance laws, provided that certain requirements are satisfied in the discretization of the non-conservative terms.
In particular, we consider PDEs of the form
\begin{equation} \label{eq:noncons-system}
\bigpartialderiv{}{t} \state{u} 
+ \bigpartialderiv{}{x} \state{f} (\state{u}) 
+ \underbrace{
    \stateG{\phi} (\state{u}) \circ \bigpartialderiv{}{x} {\state{r}(\state{u}, x)}
}_{\noncon(\state{u}, \partial_x \state{u})}
= \state{0},
\end{equation}
where $\state{u}$ denotes the vector of prognostic variables, $\state{f}$ is a (generally nonlinear) advective flux function, and $\noncon$ represents a non-conservative term. The latter can be expressed as the Hadamard (element-wise) product between the state-dependent vector $\stateG{\phi}(\state{u})$ and the gradient of $\state{r}(\state{u},x)$, a spatially varying quantity that may also depend on $\state{u}$.
A wide range of systems can be cast in this form, including the shallow water equations with variable bottom topography, the magneto-hydrodynamics (MHD) equations \cite{Powell2001,Dedner2002,Derigs2017}, the Baer–Nunziato model for multiphase flow \cite{baer1986two}, and the compressible Euler equations with gravitational effects \cite{waruszewski2022entropy,chandrashekar2017well}, among others.

\subsection{High-order DG discretization}

To approximate the solution of \eqref{eq:noncons-system}, we employ a split-form discontinuous Galerkin (DG) method in flux-differencing form based on Legendre–Gauss–Lobatto (LGL) nodes (see, e.g., \cite{Gassner2016,ranocha2021efficient}).
The computational domain is partitioned into non-overlapping elements $\Omega_e$, within which all variables are represented by piecewise Lagrange interpolating polynomials of degree $N$ defined on the LGL nodes, $\ell_k(\xi)$ with $k = 0, \ldots, N$. These basis functions are continuous inside each element but discontinuous across element interfaces.

Equation \eqref{eq:noncons-system} is then multiplied by an arbitrary test polynomial of degree $N$ and integrated by parts over each element of the mesh. The resulting integrals are evaluated using an $N+1$ point LGL quadrature rule on the reference element $\xi \in [-1,1]$. 
To provide nonlinear stability properties, such as entropy stability, the volume terms are reformulated in split form using two-point fluxes \cite{Fisher2013a,Gassner2013,Carpenter2014,Gassner2016}, leading to the expression \cite{rueda2023entropy,rueda2024flux}
\begin{empheq}{align} \label{eq:DGSEM} 
m_j \dot{\state{u}}^{\DG}_j 
+ &
\underbrace{
\sum_{k=0}^N S_{jk} \left( \state{f}^{*}_{(j,k)} + \numnonconsS{\Jan}_{(j,k)} \right)
}_{\mathrm{Volume \, term}}
- 
\underbrace{
\delta_{j0} \left( \numfluxb{f}_{(0,L)} + \numnonconsD{\Jan}_{(0,L)} \right)
+ \delta_{jN} \left( \numfluxb{f}_{(N,R)} + \numnonconsD{\Jan}_{(N,R)} \right)
}_{\mathrm{Surface \,  term}}
= \state{0},
\end{empheq}
where $m_j := J \omega_j$ denotes the $j$-th diagonal entry of the mass matrix, $J$ is the determinant of the Jacobian of the mapping from reference to physical space, $\xi \in [-1,1] \mapsto x \in \Omega_e$, $\omega_j$ is the LGL quadrature weight at node $j$, $\delta_{ij}$ is the Kronecker delta associated with nodes $i$ and $j$, and $\mat{S}$ is a skew-symmetric derivative matrix.

The derivative matrix used in the volume term, $\mat{S} \coloneqq 2\mat{Q} - \mat{B} = \mat{Q} - \mat{Q}^T$, is skew-symmetric (see, e.g., \cite{Gassner2013,ranocha2021efficient}), where the matrix $\mat{Q}$ is defined in terms of the quadrature weights and the Lagrange interpolating polynomials, $Q_{jk}=\omega_j \ell'_k(\xi_j)$, and $\mat{B} := \text{diag}(-1, 0,  \ldots, 0, 1)$ is the so-called boundary matrix.
The matrices $\mat{Q}$ and $\mat{B}$ fulfill the summation-by-parts (SBP) property \cite{Gassner2013},
\begin{equation}\label{eq:sbp_property}
    \mat{Q} + \mat{Q}^T = \mat{B},
\end{equation}
a discrete analog of integration by parts.
Moreover, they fulfill the following properties \cite{Gassner2013, Hennemann2020},
\begin{align}\label{eq:sbp_sum_row}
    \sum_{k=0}^N Q_{jk} &= 0,
    &\sum_{k=0}^N B_{jk} &= \delta_{jN} - \delta_{j0}, &\left(j=0, \ldots, N\right),
    \\
    \sum_{j=0}^N Q_{jk} &= \delta_{kN} - \delta_{k0},
    \label{eq:sbp_sum_boundary2}
    &\sum_{j=0}^N B_{jk} &= \delta_{kN} - \delta_{k0}, &\left(k=0, \ldots, N\right).
\end{align}
Finally, $\mat{Q}$, $\mat{B}$, and $\mat{S}$ are skew-centrosymmetric matrices (see, e.g., \cite{chen1996reducing}):
\begin{align}\label{eq:skew_centrosymmetry}
    Q_{jk} = - Q_{N-j,N-k},&  &B_{jk} = - B_{N-j,N-k},&  &S_{jk} = - S_{N-j,N-k}.
\end{align}

The discretization naturally separates into surface and volume contributions.
In the surface contribution, $\numfluxb{f}$ denotes the numerical flux function, while $\numnonconsD{\Jan}_{(j,k)}$ represents the numerical non-conservative surface term. Both are typically constructed from approximate Riemann solvers and therefore depend on local values as well as those from neighboring elements to the left ($L$) and right ($R$).

The volume contribution is computed using the skew-symmetric (and skew-centrosymmetric) derivative matrix $\mat{S}$ applied to two-point volume fluxes. Specifically, ${\state{f}}^{*}_{(j,k)}$ is a symmetric two-point flux function, consistent with the continuous flux, while $\numnonconsS{\Jan}_{(j,k)}$ is a generally non-symmetric two-point discretization of the non-conservative term. Appropriate choices of these numerical fluxes allow the construction of flexible split-form formulations of the nonlinear PDE terms, which mitigate aliasing errors and, in some cases, ensure provable entropy stability (see, e.g., \cite{Fisher2013,Fisher2013a,Carpenter2014,Gassner2013,Gassner2016,Renac2019}).

Depending on how the non-conservative term is discretized, the two-point flux function $\numnonconsS{\Jan}_{(j,k)}$ typically satisfies one of the following two consistency conditions:
\begin{enumerate}
\item $\numnonconsS{\Jan}_{(j,j)} = \stateG{\Phi}_j := \stateG{\phi}(\state{u}_j)\, \circ \,\state{r}_j$.
This form of consistency commonly arises in discretizations where the two-point non-conservative term can be expressed as the product of local and symmetric contributions. See, e.g., \cite{rueda2023entropy,rueda2025entropy}.
\item $\numnonconsS{\Jan}_{(j,j)} = \state{0}$.  
This condition is typical of discretizations based on either skew-symmetric two-point non-conservative fluxes, $\numnonconsS{\Jan}_{(i,j)} = -\numnonconsS{\Jan}_{(j,i)}$ (see, e.g., \cite{waruszewski2022entropy,wintermeyer2017entropy,ranocha2017shallow}), or Vol'pert-type fluctuations (see, e.g., \cite{vol1967spaces,Renac2019,coquel2021entropy}).  

\end{enumerate}

\subsection{Subcell limiting}

Following \cite{Pazner2020,RUEDARAMIREZ2022,rueda2024flux,rueda2023monolithic}, our goal is to express the DG discretization of a non-conservative system in the form of an equivalent flux-differencing formula,
\begin{equation}
    m_j \dot{\state{u}}^{\DG}_j = 
    \stateG{\Gamma}^{\DG}_{(j,j-1)}
    -\stateG{\Gamma}^{\DG}_{(j,j+1)}
    ,
    \quad
    j=0, \ldots, N.
\end{equation}
This formulation enables the combination of the so-called staggered fluxes $\stateG{\Gamma}^{\DG}_{(j,k)}$ with low-order fluxes from a finite volume (FV) method, thereby allowing the enforcement of physics admissibility and non-oscillatory properties.

A common strategy for constructing bounds-preserving schemes within a high-order DG framework is to blend DG and FV fluxes through a convex combination,  
\begin{equation}\label{eq:Subcell_blend2}  
    \stateG{\Gamma}_{(a,b)} = (1-\alpha_{(a,b)}) \stateG{\Gamma}^{\DG}_{(a,b)} + \alpha_{(a,b)} \stateG{\Gamma}^{\FV}_{(a,b)},  
\end{equation}  
which defines the hybrid DG/FV flux. Substituting this into the semi-discrete formulation yields the hybrid scheme  
\begin{equation} \label{eq:Subcell_blend}  
m_j \dot{\state{u}}_{j} =  
\left(  
  \stateG{\Gamma}_{(j,j-1)}  
- \stateG{\Gamma}_{(j,j+1)}  
\right),  
\end{equation}  
where the blending coefficient $\alpha_{(a,b)}$ is independently chosen for each interface $(a,b)$ between adjacent nodes. These coefficients can be adapted to enforce positivity, entropy conditions, or non-oscillatory behavior. 
These blending coefficients $\alpha_{(a,b)}$, also known as blending factors, determine the degree of contribution of FV relative to the high-order DG scheme.  

As discussed earlier, the flux-differencing formula for DG discretizations of non-conservative systems introduced in \cite{rueda2024flux} can be applied to well-known discretizations of the shallow-water non-conservative term, such as those of Wintermeyer et al.\ \cite{wintermeyer2017entropy} and Fjordholm et al.\ \cite{fjordholm2011well}, once these have been recast as the product of local and symmetric components.  
However, although the discretizations of \citet{wintermeyer2017entropy} and \citet{fjordholm2011well} are well-balanced for standard DG schemes, the well-balanced property is generally lost when node-wise subcell limiting is applied.  

The loss of well-balancing arises because the individual staggered fluxes in the formula of \citet{rueda2024flux} do not vanish for a well-balanced state when using the local-times-symmetric recast of the fluxes from \citet{wintermeyer2017entropy} and \citet{fjordholm2011well}. Instead, maintaining equilibrium requires that specific terms in $\stateG{\Gamma}_{(j,j-1)}$ cancel with corresponding terms in $\stateG{\Gamma}_{(j,j+1)}$.

To demonstrate the problem, we consider the flux-differencing formula from \cite{rueda2024flux} for a discretization that is individually well-balanced in both the DG and FV formulations. After reformulation in a local-times-symmetric form, the \textit{difference} between staggered fluxes that are evaluated at the lake-at-rest steady state (denoted as $\tilde{\stateG{\Gamma}}$) exactly cancel in each subcell
\begin{equation}\label{eq:wb_property_dg_fv}
	\tilde{\stateG{\Gamma}}_{(j,j-1)}^{\DG} - \tilde{\stateG{\Gamma}}_{(j,j+1)}^{\DG} = \tilde{\stateG{\Gamma}}_{(j,j-1)}^{\FV} - \tilde{\stateG{\Gamma}}_{(j,j+1)}^{\FV} = 0, \quad j=0,...,N.
\end{equation}
Now applying the hybrid scheme \eqref{eq:Subcell_blend} with blended hybrid fluxes \eqref{eq:Subcell_blend2} for the same steady-state conditions and using the well-balanced property \eqref{eq:wb_property_dg_fv} we obtain 
\begin{align}\label{eq:hybrid_flux_balance}
	\tilde{\stateG{\Gamma}}_{(j,j-1)} - \tilde{\stateG{\Gamma}}_{(j,j+1)} =& 
    \left( 1 - \alpha_{(j,j+1)} \right)\tilde{\stateG{\Gamma}}_{(j,j+1)}^{\DG} - 
    \left( 1 - \alpha_{(j,j-1)} \right)\tilde{\stateG{\Gamma}}_{(j,j-1)}^{\DG} 
    + \alpha_{(j,j-1)}\tilde{\stateG{\Gamma}}_{(j,j-1)}^{\FV} - \alpha_{(j,j+1)}\tilde{\stateG{\Gamma}}_{(j,j+1)}^{\FV},
\end{align}
where $j=0,...,N$.
In general, \eqref{eq:hybrid_flux_balance} is non-zero because unequal node-wise blending coefficients do not cancel between adjacent staggered fluxes.
Consequently, the only strategy available to date for achieving shock capturing and enforcing bounds-preserving properties in hybrid FV/DG discretizations of the shallow water system has been the less local element-wise limiting approach, as introduced in \cite{Hennemann2020,Rueda-Ramirez2020,Rueda-Ramirez2021}. In that case, an element-wise choice of the blending coefficient ($\alpha_{j,j-1} = \alpha_{j,j+1}$), together with the well-balanced property \eqref{eq:wb_property_dg_fv} guarantees exact cancellation of the hybrid staggered fluxes in \eqref{eq:hybrid_flux_balance}.

On the other hand, if well-balancedness is to hold for arbitrary nodal blending coefficients, each staggered DG flux must vanish individually at the node level. 
It is clear that the staggered fluxes introduced in \citet{rueda2024flux} are not applicable for this purpose, as the local-times-symmetric formulation contains no mechanism that enforces node-wise cancellation. Instead, we will demonstrate that a well-balanced method with node-wise subcell-limiting can be achieved by combining a novel flux-differencing formula with a specific reformulation of the shallow water equations.

\section{How to fix subcell limiting} \label{sec:fix-limiting}

As discussed above, enabling a nodal selection of the blending coefficients requires staggered DG fluxes that vanish individually at each interface for a well-balanced state. 
In other words, the well-balancedness must hold locally at the node level, without relying on cancellations between neighboring interfaces.

To achieve this property, we introduce a novel flux-differencing formulation based on a local-times-jump discretization of the non-conservative term. 
The key idea is to construct a non-conservative numerical contribution that is directly proportional to the jump of the total water height,
$$
H = h + b,
$$
where $h$ denotes the fluid depth and $b$ the bottom topography.

More precisely, we seek a formulation in which the numerical non-conservative term is proportional to
$$
\jump{H}_{(L,R)} := H_R - H_L.
$$
For a well-balanced state (e.g., lake-at-rest), the total height is constant and thus $\jump{H}_{(L,R)} = 0$. 
Consequently, each staggered flux vanishes individually at equilibrium, ensuring node-wise well-balancedness and making subcell limiting with arbitrary nodal blending coefficients admissible.

Obtaining such a structure requires a suitable reformulation of the shallow water equations. This reformulation is presented in Section~\ref{sec:reformulation_swe}. 
Based on it, we then derive a compatible flux-differencing formula in Section~\ref{sec:fluxdiff}, which guarantees the desired local cancellation property.

\subsection{Reformulation of the shallow water equations}\label{sec:reformulation_swe}

In order to construct a non-conservative term that explicitly involves the derivative of the total water height, $H = h + b$, we first rewrite the shallow water equations in a suitable form. 
The resulting formulation will later allow us to design a flux-differencing discretization in which the non-conservative contribution is proportional to the jump in $H$.

We consider the shallow water system \eqref{eq:swe_std} in one spatial dimension and rewrite it as
\begin{equation}\label{eq:swe}
    \bigpartialderiv{}{t}
    \begin{pmatrix}
        h \\
        h v
    \end{pmatrix}
    +
    \bigpartialderiv{}{x}
    \begin{pmatrix}
        h v \\
        h v^2
    \end{pmatrix}
    +
    g h \, 
    \bigpartialderiv{}{x}
    \begin{pmatrix}
        0 \\
        h + b
    \end{pmatrix}
    =
    \begin{pmatrix}
        0 \\
        0
    \end{pmatrix},
\end{equation}
where $h$ denotes the fluid depth, $v$ the velocity, $g$ the gravitational acceleration, and $b$ the bottom topography.

To obtain this representation, we employ a split form of the pressure term, which is commonly written in conservative form inside the flux divergence. 
In particular, we use the identity
$$
\bigpartialderiv{}{x} \left( \frac{1}{2} g h^2 \right)
=
g h \, \bigpartialderiv{h}{x},
$$
so that the pressure contribution can be expressed as a product of $g\,h$ and the gradient of $h$.
This manipulation enables us to combine the pressure gradient and the topography source term into a single term proportional to $\partial_x (h + b)$, i.e., the derivative of the total water height.

A closely related reformulation was previously introduced by \citet{ersing2025entropy} in the context of the multilayer shallow water equations, where it was used to construct entropy-stable and well-balanced DG discretizations.
For a single fluid layer in one spatial dimension, the entropy-conserving and well-balanced two-point fluxes proposed by \citet{ersing2025entropy} reduce to
\begin{equation}
\begin{aligned}\label{eq:ersing_fluxes}
    \state{f}^{*}_{(j,k)} = & 
    \begin{pmatrix}
        \avg{h v}_{(j,k)} \\
        \avg{h v}_{(j,k)} \avg{v}_{(j,k)}
    \end{pmatrix},
    \quad
    \numnonconsS{\Jan}_{(j,k)} = &
    \begin{pmatrix}
        0 \\
        \frac{1}{2} g h_j \jump{h + b}_{(j,k)}
    \end{pmatrix},
\end{aligned}
\end{equation}
where $\avg{\cdot}_{(j,k)}$ denotes the arithmetic average between the states at nodes $j$ and $k$.

When the two-point fluxes \eqref{eq:ersing_fluxes} are inserted into the split-form DGSEM formulation \eqref{eq:DGSEM}, the resulting semi-discrete scheme is both entropy stable and well-balanced. In particular, for a lake-at-rest equilibrium, characterized by constant total height and vanishing velocity, the jump $\jump{h+b}_{(j,k)}$ is zero for every two-point connection. 
Consequently, both the conservative and non-conservative contributions satisfy $\state{f}^{*}_{(j,k)} = \numnonconsS{\Jan}_{(j,k)} = \state{0}$ locally, i.e., at the level of each individual node pair.

\subsection{A new flux-differencing formula} \label{sec:fluxdiff}

In this section, we introduce a novel flux-differencing formula designed to vanish locally for equilibrium states of systems of balance laws.  
The new formula shares several features with the approach of \citet{rueda2024flux}, but requires the non-conservative two-point flux to take a specific form: it must be expressed as the product of a local contribution and the jump of a relevant physical quantity.  
To ensure equilibrium preservation (i.e., well-balancing), the latter term is required to vanish when the system is in an equilibrium state, as is the case for \eqref{eq:ersing_fluxes}.

Given the similarities between the new subcell limiting formula and the one proposed in \cite{rueda2024flux}, we highlight the new terms in \change{orange} to facilitate direct comparison.  

\begin{proposition}\label{prop:fluxdiff}
It is possible to rewrite \eqref{eq:DGSEM} as a flux-differencing formula,
\begin{equation} \label{eq:fluxdiff}
    m_j \dot{\state{u}}^{\DG}_j = 
    \stateG{\Gamma}^{\DG}_{(j,j-1)}
    -\stateG{\Gamma}^{\DG}_{(j,j+1)},
    \quad
    j=0, \ldots, N,
\end{equation}
where the indices $j=-1$ and $j=N+1$ refer to the outer states (across the left and right boundaries, respectively) and $\stateG{\Gamma}^{\DG}_{(j,k)}$ is the so-called staggered (or telescoping) ``flux'' between node $j$ and the \textbf{adjacent} node $k$,
if it is possible to write the volume numerical non-conservative term as a product of a local contribution and a jump term,
\begin{equation} \label{eq:condition}
    \numnonconsS{\Jan}_{(j,k)} := \Jan^{\mathrm{loc}}_j \circ \, \change{\Jan^{\sjump}_{(j,k)}},
\end{equation}
where $\Jan^{\mathrm{loc}}_j := \Jan^{\mathrm{loc}} (\state{u}_j)$ only depends on local quantities, and  $\Jan^{\sjump}_{(j,k)} = -\Jan^{\sjump}_{(k,j)}$ is defined as the jump of some quantity between nodes $j$ and k:
\begin{equation}\label{eq:skew_as_jump}
    \Jan^{\sjump}_{(j,k)} = \jump{\state{r}}_{(j,k)} = \state{r}_k - \state{r}_j.
\end{equation}

The staggered fluxes are then defined as
\begin{align}
\stateG{\Gamma}^{\DG}_{(0,-1)}  =& \ \numfluxb{f}_{(0,L)} + \numnonconsD{\Jan}_{(0,L)}, 
\label{eq:leftFlux}
\\
\stateG{\Gamma}^{\DG}_{(j,k)} =&  \sum_{l=0}^{\min(j,k)} \sum_{m=0}^N S_{lm} \state{f}^{*}_{(l,m)}+\Jan^{\mathrm{loc}}_j \circ \sum_{l=0}^{\min(j,k)} \sum_{m=0}^N S_{lm} \Jan^{\sjump}_{(l,m)} 
\change{+ 2 \Jan^{\mathrm{loc}}_j \circ \Jan^{\sjump}_{(j,0)}}, 
& \nonumber\\
&\forall (j,k) = (j,j \pm 1) \setminus \{(0,-1), (N,N+1) \},
\label{eq:internFlux}\\
\stateG{\Gamma}^{\DG}_{(N,N+1)} =& \ \numfluxb{f}_{(N,R)} +  \numnonconsD{\Jan}_{(N,R)}.
\label{eq:rightFlux}
\end{align}
\end{proposition}

\begin{proof}
It suffices to evaluate \eqref{eq:fluxdiff} for the boundary nodes, $j=0$ and $j=N$, and for an internal node $j \notin \{0,N\}$.

For the left boundary, $j=0$, we obtain
\begin{align*}
    m_0 \dot{\state{u}}^{\DG}_0 &= 
    \stateG{\Gamma}^{\DG}_{(0,-1)}
    -\stateG{\Gamma}^{\DG}_{(0,1)} \\
    &=
    \numfluxb{f}_{(0,L)} + \numnonconsD{\Jan}_{(0,L)}
    - 
    \sum_{m=0}^N S_{0m} \state{f}^{*}_{(0,m)} 
    - \Jan^{\mathrm{loc}}_0 \circ     \sum_{m=0}^N S_{0m}     \Jan^{\sjump}_{(0,m)}
    \\
\mathrm{(using~eq.~\eqref{eq:condition})} \qquad
    &=
    \numfluxb{f}_{(0,L)} + \numnonconsD{\Jan}_{(0,L)}
    - 
    \sum_{m=0}^N S_{0m} \left( \state{f}^{*}_{(0,m)} + \numnonconsS{\Jan}_{(0,m)} \right),
\end{align*}
which is equivalent to \eqref{eq:DGSEM} for $j=0$.

For an arbitrary internal degree of freedom, $j \notin \{0,N\}$, we obtain
\begin{align*}
    m_j \dot{\state{u}}^{\DG}_j =\ &
    \stateG{\Gamma}^{\DG}_{(j,j-1)}
    -\stateG{\Gamma}^{\DG}_{(j,j+1)}
    \\
    =&
    \sum_{l=0}^{j-1} \sum_{m=0}^N S_{lm} \state{f}^{*}_{(l,m)} 
    + \Jan^{\mathrm{loc}}_j \circ \sum_{l=0}^{j-1} \sum_{m=0}^N S_{lm} \Jan^{\sjump}_{(l,m)}
    - \sum_{l=0}^{j} \sum_{m=0}^N S_{lm} \state{f}^{*}_{(l,m)} 
    - \Jan^{\mathrm{loc}}_j \circ \sum_{l=0}^{j} \sum_{m=0}^N S_{lm} \Jan^{\sjump}_{(l,m)}
    \\
    =&
    - \sum_{m=0}^N S_{jm} \state{f}^{*}_{(j,m)} 
    - \Jan^{\mathrm{loc}}_j \circ \sum_{m=0}^N S_{jm} \Jan^{\sjump}_{(j,m)}
    \\
\mathrm{(using~eq.~\eqref{eq:condition})} \qquad
    =&
    -
    \sum_{m=0}^N S_{jm} \left( \state{f}^{*}_{(j,m)} + \numnonconsS{\Jan}_{(j,m)} \right),
\end{align*}
which is the desired result.

For the right boundary, $j=N$, we obtain
\begin{align}
    m_N \dot{\state{u}}^{\DG}_N =& \ 
    \stateG{\Gamma}^{\DG}_{(N,N-1)}
    -\stateG{\Gamma}^{\DG}_{(N,N+1)} \\
    =& \ 
    \change{2 \Jan^{\mathrm{loc}}_N \circ \Jan^{\sjump}_{(N,0)}}\nonumber\\
    &+ \sum_{l=0}^{N-1} \sum_{m=0}^N S_{lm} \state{f}^{*}_{(l,m)} 
+ \Jan^{\mathrm{loc}}_N \circ \sum_{l=0}^{N-1} \sum_{m=0}^N S_{lm} \Jan^{\sjump}_{(l,m)}
    \nonumber\\
    &- \numfluxb{f}_{(N,R)}  
    - \numnonconsD{\Jan}_{(N,R)}
    \\
    =& \ 
    \change{2 \Jan^{\mathrm{loc}}_N \circ \Jan^{\sjump}_{(N,0)}}+
    \nonumber\\
    &\sum_{l=0}^{N} \sum_{m=0}^N S_{lm} \state{f}^{*}_{(l,m)} 
    + \Jan^{\mathrm{loc}}_N \circ \sum_{l=0}^{N} \sum_{m=0}^N S_{lm} \Jan^{\sjump}_{(l,m)}
    \nonumber\\
    &-\sum_{m=0}^N S_{Nm} \state{f}^{*}_{(N,m)} 
    - \Jan^{\mathrm{loc}}_N \circ \sum_{m=0}^N S_{Nm} \Jan^{\sjump}_{(N,m)}
    \nonumber\\
    &- \numfluxb{f}_{(N,R)}  
    - \numnonconsD{\Jan}_{(N,R)}
    \\
\mathrm{(Skew-symmetry~of~\mat{S}~\&~eq.~\eqref{eq:condition})} \qquad
    =& \ 
    \change{2 \Jan^{\mathrm{loc}}_N \circ \Jan^{\sjump}_{(N,0)}}\nonumber\\
    &+\frac{1}{2} \sum_{l=0}^{N} \sum_{m=0}^N (S_{lm} - S_{ml}) \state{f}^{*}_{(l,m)} 
    + \Jan^{\mathrm{loc}}_N \circ \sum_{l=0}^{N} \sum_{m=0}^N S_{lm} \Jan^{\sjump}_{(l,m)}
    \nonumber\\
    &-\sum_{m=0}^N S_{Nm} \left( \state{f}^{*}_{(N,m)} + \Jan^{\star}_{(N,m)} \right)
    - \numfluxb{f}_{(N,R)}  
    - \numnonconsD{\Jan}_{(N,R)}
    \label{eq:skew_symmetry_S}
    \\
\mathrm{(re-index~\&~symmetry~of~\state{f}^*)} \qquad
    =& \ \change{2 \Jan^{\mathrm{loc}}_N \circ \Jan^{\sjump}_{(N,0)}}
    + \Jan^{\mathrm{loc}}_N \circ \sum_{l=0}^{N} \sum_{m=0}^N S_{lm} \Jan^{\sjump}_{(l,m)}
    \nonumber\\
    &-\sum_{m=0}^N S_{Nm} \left( \state{f}^{*}_{(N,m)} + \Jan^{\star}_{(N,m)} \right)
    - \numfluxb{f}_{(N,R)}  
    - \numnonconsD{\Jan}_{(N,R)}, \label{eq:symmetry_f_vanishes}
\end{align}
which is again exactly equivalent to \eqref{eq:DGSEM} for $j=N$, as the first line of the last right-hand side of \eqref{eq:symmetry_f_vanishes} vanishes due to:
\begin{align}
    \sum_{l=0}^{N} \sum_{m=0}^N S_{lm} \Jan^{\sjump}_{(l,m)}
    =&
    \sum_{l=0}^{N} \sum_{m=0}^N S_{lm} \left( \state{r}_m - \state{r}_l \right)
    \nonumber\\
    =&
    \sum_{m=0}^{N} \state{r}_m  \sum_{l=0}^N S_{lm} 
    -
    \sum_{l=0}^{N} \state{r}_l  {\sum_{m=0}^N S_{lm}}
    \nonumber\\
\mathrm{(Definition~of~\mat{S})} \qquad
    =&
    \sum_{m=0}^{N} \state{r}_m \left( 2\sum_{l=0}^N Q_{lm}
    - \sum_{l=0}^N B_{lm} \right)
    -
    \sum_{l=0}^{N} \state{r}_l  \left({2\sum_{m=0}^N Q_{lm}}
    -  {\sum_{m=0}^N B_{lm}}
    \right)
    \nonumber\\
\mathrm{(eqs.~\eqref{eq:sbp_sum_row}~\&~\eqref{eq:sbp_sum_boundary2})} \qquad
    =&
    \sum_{m=0}^{N} \state{r}_m \left( 2 
    \underbrace{\sum_{l=0}^N Q_{lm}}_{:= \delta_{mN} - \delta_{m0}}
    - \underbrace{\sum_{l=0}^N B_{lm}}_{:= \delta_{mN} - \delta_{m0}} \right)
    -
    \sum_{l=0}^{N} \state{r}_l  \left({2\cancel{\sum_{m=0}^N Q_{lm}}}
    -  \underbrace{\sum_{m=0}^N B_{lm}}_{:= \delta_{lN} - \delta_{l0}}
    \right)
    \nonumber\\
    =& \ 
    2 \left(\state{r}_N - \state{r}_0 \right)
    \nonumber\\
    =& \ 
    2 \Jan^{\sjump}_{(0,N)}.
    \nonumber\\
    =&
    -2 \Jan^{\sjump}_{(N,0)}.
    \label{eq:sum_SPhijump}
\end{align}

\end{proof}

\begin{corollary}\label{cor:nodewise_balance}
    The main advantage of the staggered DG fluxes defined in \eqref{eq:leftFlux}-\eqref{eq:rightFlux} is that, besides being consistent with the high-order DGSEM formulation \eqref{eq:DGSEM}, they vanish for equilibrium states satisfying 
$$
\state{f}^{*}_{(j,k)} =\Jan^{\sjump}_{(j,k)} = \state{0}.
$$
Consequently, every staggered contribution in \eqref{eq:fluxdiff} is zero at equilibrium, and the discrete scheme preserves such states exactly.
\end{corollary}

The following remarks are in order.
\begin{remark}
In complete analogy to the staggered fluxes proposed by \citet{rueda2024flux}, in the purely conservative case, i.e., when $\noncon = \stateG{\Phi} = \state{0}$, the expressions \eqref{eq:leftFlux}--\eqref{eq:rightFlux} reduce to the telescoping fluxes introduced by \citet{Fisher2013a}. These fluxes are symmetric in their indices,
$$
\stateG{\Gamma}^{\DG}_{(j,k)} = \stateG{\Gamma}^{\DG}_{(k,j)},
$$
and therefore define a locally conservative discretization at the node level.
\end{remark}

\begin{remark}\label{remark:alternative}
There exist alternative flux-differencing formulations for non-conservative terms in local-times-jump form which, when inserted into \eqref{eq:fluxdiff}, also recover the high-order DGSEM discretization \eqref{eq:DGSEM}. 
For instance, consider the staggered DG fluxes
    \begin{align}
\stateG{\Gamma}^{a}_{(0,-1)}  =& \ \numfluxb{f}_{(0,L)} + \numnonconsD{\Jan}_{(0,L)},
\label{eq:leftFlux_alternative}
\\
\stateG{\Gamma}^{a}_{(j,k)} =& \sum_{l=0}^{\min(j,k)} \sum_{m=0}^N S_{lm} \state{f}^{*}_{(l,m)}+\Jan^{\mathrm{loc}}_j \circ \sum_{l=0}^{\min(j,k)} \sum_{m=0}^N S_{lm} \Jan^{\sjump}_{(l,m)}, 
& \nonumber
\\
& \forall (j,k) \setminus \{(0,-1), (N,N+1), (N,N-1) \},
\\
\stateG{\Gamma}^{a}_{(N,N-1)} =&  
     \sum_{l=0}^{N-1} \sum_{m=0}^N S_{lm} \state{f}^{*}_{(l,m)}
    +\Jan^{\mathrm{loc}}_N \circ \sum_{l=0}^{N-1} \sum_{m=0}^N S_{lm} \Jan^{\sjump}_{(l,m)} 
\change{+ 2 \Jan^{\mathrm{loc}}_N \circ \Jan^{\sjump}_{(N,0)}}, 
\\
\stateG{\Gamma}^{a}_{(N,N+1)} =& \ \numfluxb{f}_{(N,R)} +  \numnonconsD{\Jan}_{(N,R)}.
\label{eq:rightFlux_alternative}
\end{align}
Repeating the arguments in the proof of Proposition~\ref{prop:fluxdiff} shows that these fluxes also yield the high-order DGSEM formulation \eqref{eq:DGSEM} when substituted into \eqref{eq:fluxdiff}. 

However, as demonstrated in Section~\ref{sec:Application}, this alternative choice leads to asymmetric numerical solutions, even for symmetric initial and boundary data, and may produce spurious over- and undershoots. A more detailed discussion is provided in Section~\ref{sec:blast} and Remark~\ref{remark:criteria}.
\end{remark}

\begin{remark}\label{remark:criteria}
    Extensive numerical experiments suggest that suitable staggered DG fluxes should satisfy a property that we refer to as \emph{pairwise symmetry at interfaces}. 
    Specifically, for every pair of nodes $(j,k)$, the staggered fluxes $\stateG{\Gamma}_{(j,k)}$ and $\stateG{\Gamma}_{(k,j)}$ should be constructed from the same algebraic expression, differing only through the interchange of the node indices. 
    In particular, no additional correction terms should appear in one orientation but not the other.

    This property ensures consistency between the staggered flux evaluated from node $j$ to node $k$ and that evaluated from node $k$ to node $j$, thereby preventing orientation-dependent behavior at shared interfaces.
    Such symmetry appears to be essential for obtaining robust and non-oscillatory numerical solutions.

    It is straightforward to verify that the staggered fluxes defined in \eqref{eq:leftFlux}--\eqref{eq:rightFlux} satisfy pairwise symmetry at interfaces.
    In contrast, the alternative formulation introduced in Remark~\ref{remark:alternative}, given by \eqref{eq:leftFlux_alternative}--\eqref{eq:rightFlux_alternative}, violates this property because $\stateG{\Gamma}^a_{(N-1,N)}$ and $\stateG{\Gamma}^a_{(N,N-1)}$ are computed using different algebraic expressions.
\end{remark}
\begin{remark}
    The staggered fluxes introduced in this work satisfy a property that we term \emph{invariance under index relabeling}.
    Specifically, consider reversing the nodal ordering within an element from $\{0,\ldots,N\}$ to $\{N,\ldots,0\}$ (see Figure~\ref{fig:labeling}). 
    Under this relabeling, the staggered DG fluxes remain unchanged.
    In other words, their values are independent of the orientation of the local node numbering.
    This property ensures that the resulting multidimensional discretization is independent of the particular indexing induced by mesh generation, thereby eliminating orientation-dependent artifacts.

    To verify this property for the fluxes defined in \eqref{eq:leftFlux}--\eqref{eq:rightFlux}, we rewrite \eqref{eq:internFlux} using a right-to-left indexing convention. 
    Figure~\ref{fig:labeling} illustrates the corresponding relabeling,
$$
i \leftarrow (N-j), \qquad
q \leftarrow (N-k), \qquad
p \leftarrow (N-l), \qquad
n \leftarrow (N-m),
$$
which induces a transformed storage of the fluxes and non-conservative terms. 
We then manipulate the expression for $\Gamma_{(i,q)}$ under right-to-left indexing and recover the staggered flux definition associated with left-to-right indexing.

To this end, all quantities except for the skew-symmetric differentiation matrix $\mat{S}$ are reindexed, yielding
    \begin{figure}
        \centering
        \begin{tikzpicture}[
    scale=1.2,
    >=latex,
    font=\small,
    every node/.style={inner sep=1pt}
]

\def\L{4.5}      
\def\gap{1.8}    
\def\y{0}

\pgfmathsetmacro{\shift}{\L+\gap}

\def\xA{0.00}
\def\xB{0.78}
\def\xC{2.25}
\def\xD{3.72}
\def\xE{4.50}

\draw[line width=0.9pt] (0,\y) -- (\L,\y);

\foreach \x/\label in {
    \xA/0,
    \xB/1,
    \xC/2,
    \xD/\ldots,
    \xE/N
}{
    \filldraw[black] (\x,\y) circle (2pt);
    \node[below=4pt] at (\x,\y) {\label};
}

\node[align=center] at (2.25,1.25)
{\textbf{Left to right ($L \rightarrow R$)}\\
$j,k,l,m$ indices};

\begin{scope}[shift={(\shift,0)}]

    \draw[line width=0.9pt] (0,\y) -- (\L,\y);

    \foreach \x/\label in {
        \xA/N,
        \xB/\ldots,
        \xC/2,
        \xD/1,
        \xE/0
    }{
        \filldraw[black] (\x,\y) circle (2pt);
        \node[below=4pt] at (\x,\y) {\label};
    }

    \node[align=center] at (2.25,1.25)
    {\textbf{Right to left ($R \rightarrow L$)}\\
    $i,q,p,n$ indices};

\end{scope}

\draw[densely dashed, gray]
(\shift-0.9,-0.6) -- (\shift-0.9,1.8);

\end{tikzpicture}
        \caption{Two different labeling possibilities for the Legendre-Gauss-Lobatto nodes in one dimension.}
        \label{fig:labeling}
    \end{figure}
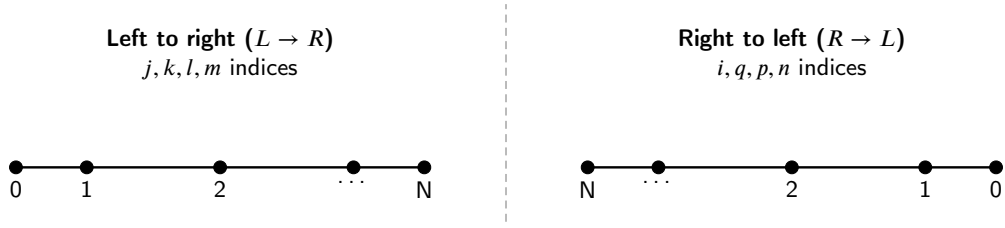
    \begin{align}
        \stateG{\Gamma}^{\DG,R\rightarrow L}_{(i,q)} =&  
        -\sum_{p=0}^{\min(i,q)} \sum_{n=0}^N S_{pn} \state{f}^{*}_{(p,n)}
        -\Jan^{\mathrm{loc}}_i \circ \sum_{p=0}^{\min(i,q)} \sum_{n=0}^N S_{pn} \Jan^{\sjump}_{(p,n)} 
        - 2 \Jan^{\mathrm{loc}}_i \circ \Jan^{\sjump}_{(i,0)}
        \label{eq:RtoLindexing}
        \\
        \mathrm{(Convert~to~}L\rightarrow R\mathrm{~indexing)} \qquad
        =&
        -\sum_{l=N}^{\max(j,k)} \sum_{m=N}^0 S_{N-l,N-m} \state{f}^{*}_{(l,m)}
        -\Jan^{\mathrm{loc}}_j \circ \sum_{l=N}^{\max(j,k)} \sum_{m=N}^0 S_{N-l,N-m} \Jan^{\sjump}_{(l,m)} 
        \nonumber \\
        & \ - 2 \Jan^{\mathrm{loc}}_j \circ \Jan^{\sjump}_{(j,N)},
        \\
        \mathrm{(Skew\hspace{-0.1cm}-\hspace{-0.1cm}centrosym.~of~\mat{S}~ \eqref{eq:skew_centrosymmetry})} \qquad
        =&
        \sum_{l=\max(j,k)}^{N} \sum_{m=0}^N S_{lm} \state{f}^{*}_{(l,m)}
        +\Jan^{\mathrm{loc}}_j \circ \sum_{l=\max(j,k)}^{N} \sum_{m=0}^N S_{lm} \Jan^{\sjump}_{(l,m)} 
        - 2 \Jan^{\mathrm{loc}}_j \circ \Jan^{\sjump}_{(j,N)},
        \\
        \mathrm{(Complete~sum~\&~\eqref{eq:skew_symmetry_S},\eqref{eq:symmetry_f_vanishes})} \qquad
        =&  
        -\sum_{l=0}^{\min(j,k)} \sum_{m=0}^N S_{lm} \state{f}^{*}_{(l,m)}
        -\Jan^{\mathrm{loc}}_j \circ \sum_{l=0}^{\min(j,k)} \sum_{m=0}^N S_{lm} \Jan^{\sjump}_{(l,m)}  
        - 2 \Jan^{\mathrm{loc}}_j \circ \Jan^{\sjump}_{(j,N)}
        \nonumber\\
        &+\underbrace{\sum_{l=0}^{N} \sum_{m=0}^N S_{lm} \state{f}^{*}_{(l,m)}}_{=0}
        + \ \Jan^{\mathrm{loc}}_j \circ \sum_{l=0}^{N} \sum_{m=0}^N S_{lm} \Jan^{\sjump}_{(l,m)},
        \\
        \mathrm{(Eq.~\eqref{eq:sum_SPhijump})} \qquad
        =&  
        -\sum_{l=0}^{\min(j,k)} \sum_{m=0}^N S_{lm} \state{f}^{*}_{(l,m)}
        -\Jan^{\mathrm{loc}}_j \circ \sum_{l=0}^{\min(j,k)} \sum_{m=0}^N S_{lm} \Jan^{\sjump}_{(l,m)}  
        - 2 \Jan^{\mathrm{loc}}_j \circ \Jan^{\sjump}_{(j,N)}
        \nonumber\\
        &
        - 2 \Jan^{\mathrm{loc}}_j \circ  \Jan^{\sjump}_{(N,0)},
        \\
        =&  
        -\sum_{l=0}^{\min(j,k)} \sum_{m=0}^N S_{lm} \state{f}^{*}_{(l,m)}
        -\Jan^{\mathrm{loc}}_j \circ \sum_{l=0}^{\min(j,k)} \sum_{m=0}^N S_{lm} \Jan^{\sjump}_{(l,m)}  
        - 2 \Jan^{\mathrm{loc}}_j \circ  \Jan^{\sjump}_{(j,0)},
        \\
        \mathrm{(Eq.~\eqref{eq:internFlux})} \qquad
        =& \ 
        -\stateG{\Gamma}^{\DG}_{(j,k)}.
    \end{align}
    In \eqref{eq:RtoLindexing}, all terms are multiplied with a minus sign because the fluxes and non-conservative terms are evaluated in a reference space direction that is opposite to the physical coordinate direction.
    When the right-to-left staggered fluxes are inserted into the the corresponding  right-to-left flux-differencing formula, the DGSEM discretization is recovered.
    
    Both the staggered fluxes introduced in \eqref{eq:leftFlux}--\eqref{eq:rightFlux} and the alternative formulation proposed in Remark~\ref{remark:alternative}, given by \eqref{eq:leftFlux_alternative}--\eqref{eq:rightFlux_alternative}, satisfy invariance under index relabeling.

\end{remark}

\section{Numerical Results}\label{sec:Application}

We apply the newly derived node-wise limiting strategy to the two-dimensional shallow water equations with non-constant bottom topography.
The spatial discretization is based on a high-order split-form DGSEM equipped with subcell finite volume limiting capabilities.
Details of the two-dimensional extension of the general numerical framework, along with the specific numerical fluxes and non-conservative term discretizations employed in the present subcell strategy, are provided in Appendices~\ref{app:2d} and \ref{app:swe_2d}, respectively.

We have made available this variant of the DGSEM and stable limiting in the open-source solvers Trixi.jl
\cite{ranocha2022adaptive,schlottkelakemper2021purely} and TrixiShallowWater.jl \cite{winters2025trixi}.
For time integration we use CFL-based time stepping with the three-stage, third-order explicit strong stability preserving (SSP) Runge-Kutta method of Shu and Osher \cite{shu1988efficient}.
The unstructured quadrilateral meshes were constructed
with gmsh \cite{geuzaine2009gmsh} and HOHQMesh \cite{kopriva2024hohqmesh:joss,kopriva2024hohqmeshjl}.
We use ParaView \cite{ahrens2005paraview} 
and Makie.jl \cite{Makie2021} to visualize the results.
The source code needed to reproduce the numerical experiments is
available online in our reproducibility repository~\cite{rueda2026numericalRepro}.

\subsection{Convergence test}
First, we demonstrate that the novel flux-differencing formula \eqref{eq:fluxdiff} preserves the spatial order of accuracy of the corresponding DGSEM formulation \eqref{eq:DGSEM}. 
To this end we apply the method of manufactured solutions, where the smooth and exact reference solution
\begin{equation}
    \begin{aligned}
        H_{\text{MS}}(\vec{x},t) &= 4 + 0.2\cos\left(\pi x + t\right) + 0.2\cos\left(\pi y + t\right),\\
        v_{1,\text{MS}}(\vec{x}, t) &= v_{2,\text{MS}}(\vec{x}, t) =  0.5 ,\\
        b(\vec{x})_{\text{MS}} &= 1 + 0.2\cos\left(\pi x \right) + 0.2\cos\left(\pi y \right),
    \end{aligned}
\end{equation}
is prescribed and substituted into the system \eqref{eq:swe2d} to derive a corresponding source term that is computed symbolically.

The flux-differencing DGSEM formulation (see \eqref{eq:fluxdiff} for the one-dimensional case and \eqref{eq:fluxdiff2d} for the two-dimensional case) is evolved in time over the domain $\Omega = [-1,1]^2$ using a sequence of curvilinear meshes. 
These meshes are generated by applying the mapping
\begin{equation}
    \vec{X} (\xi, \eta) = \begin{pmatrix}
        \xi + 0.1\,\sin\left(\pi \eta\right)\cos\left(0.5\pi \xi\right) \\
        \eta + 0.1\,\sin\left(\pi \xi\right)\cos\left(0.5\pi \eta\right)
    \end{pmatrix}
\end{equation}
to a Cartesian mesh defined on our computational domain $\Omega$, followed by interpolation of the mapped node locations using a polynomial mapping of degree~$3$. The mesh sequence consists of Cartesian grids with $N_{\mathrm{elem}}=\{16,\,64,\,256,\,1024\}$ elements.

The numerical solutions are computed using a polynomial degree $N=3$ and a fixed time step $\Delta t = 5 \cdot 10^{-4}$. 
We impose periodic boundary conditions and set the gravitational acceleration to $g = 9.81$.

Table~\ref{tab:convergence:results} reports the $L^2$-errors between the numerical and exact solutions at the final time $t = 0.1$, together with the corresponding experimental orders of convergence (EOC) for all solution variables.
The results confirm that the proposed flux-differencing formulation achieves the expected fourth-order accuracy of the DGSEM scheme (i.e., order $N+1$), thereby validating the accuracy and consistency of the formulations in \eqref{eq:fluxdiff} and \eqref{eq:fluxdiff2d}.

\begin{table}[htb]
    \caption{Convergence results obtained for polynomial degree $3$ at final time $t=0.1$ with fixed time step $\Delta t = 5\cdot{10}^{-4}$}
    \label{tab:convergence:results}
    \centering
    \begin{tabular}{c|cc|cc|cc}
        \toprule
        $N_{elem}$ & $\|h - h_{\text{MS}}\|_{L^2}$ & EOC($h$) & $\|hv_1 - hv_{1,\text{MS}}\|_{L^2}$ & EOC($hv_1$) & $\|hv_2 - hv_{2,\text{MS}}\|_{L^2}$ & EOC($hv_2$) \\
        \midrule
        $16$ 
        & $8.73\cdot10^{-4}$ & -- 
        & $3.02\cdot10^{-3}$ & -- 
        & $3.02\cdot10^{-3}$ & -- \\
        
        $64$
        & $5.59\cdot10^{-5}$ & 3.97 
        & $1.98\cdot10^{-4}$ & 3.93 
        & $1.98\cdot10^{-4}$ & 3.93 \\
        
        $256$
        & $3.38\cdot10^{-6}$ & 4.05 
        & $1.21\cdot10^{-5}$ & 4.04 
        & $1.21\cdot10^{-5}$ & 4.04 \\
        
        $1024$ 
        & $2.12\cdot10^{-7}$ & 4.00 
        & $6.72\cdot10^{-7}$ & 4.17 
        & $6.72\cdot10^{-7}$ & 4.17 \\
        \bottomrule
    \end{tabular}
\end{table}

\subsection{Well-balancedness test}
Next, we verify the well-balanced property of the proposed flux-differencing formulation introduced in Section~\ref{sec:fluxdiff} within the hybrid DG/FV framework, combined with the reformulated fluxes of \citet{ersing2025entropy}. To this end, we consider the following initial condition:
\begin{equation}
    H(\vec{x},0) = 0.45, \quad \vec{v}(\vec{x},0) = \vec{0}, \quad b(\vec{x}) = \begin{cases}
        0.2\left(1 + \cos \left(2.5\pi \|\vec{x}\|_2\right)\right) 
    & \mathrm{if} \|\vec{x}\|_2 \leq 0.4, \\
        0 & \mathrm{otherwise},
    \end{cases}
\end{equation}
which corresponds to a lake-at-rest equilibrium over a smooth bottom topography on the domain $\Omega = [-1,1]^2$.

We impose slip-wall boundary conditions and set the gravitational acceleration to $g=9.81$. 
Numerical results are computed on the curvilinear mesh shown in Figure~\ref{fig:wb:mesh}, consisting of 96 elements. 
Both the geometric mapping of the element edges and the DG approximation space employ polynomial degree $N=3$. 
We set CFL$=0.9$ and integrate until the final time $t=10$.

To rigorously assess well-balancedness in the hybrid DG/FV setting, we prescribe random node-wise blending coefficients, shown in Figure~\ref{fig:wb:alphas}, and compute the interface blending coefficients between two nodes as the maximum between the nodal coefficients,
\begin{equation}\label{eq:interface_alpha}
    \alpha_{(L,R)} = \max(\alpha_L, \alpha_R).
\end{equation}
The blending coefficients are kept constant over the entire simulation.

\begin{figure}[]
    \centering
    \begin{subfigure}{0.32\textwidth}
        \centering
        \includegraphics[width=\linewidth, trim= 4cm 0cm 7cm 0cm , clip]{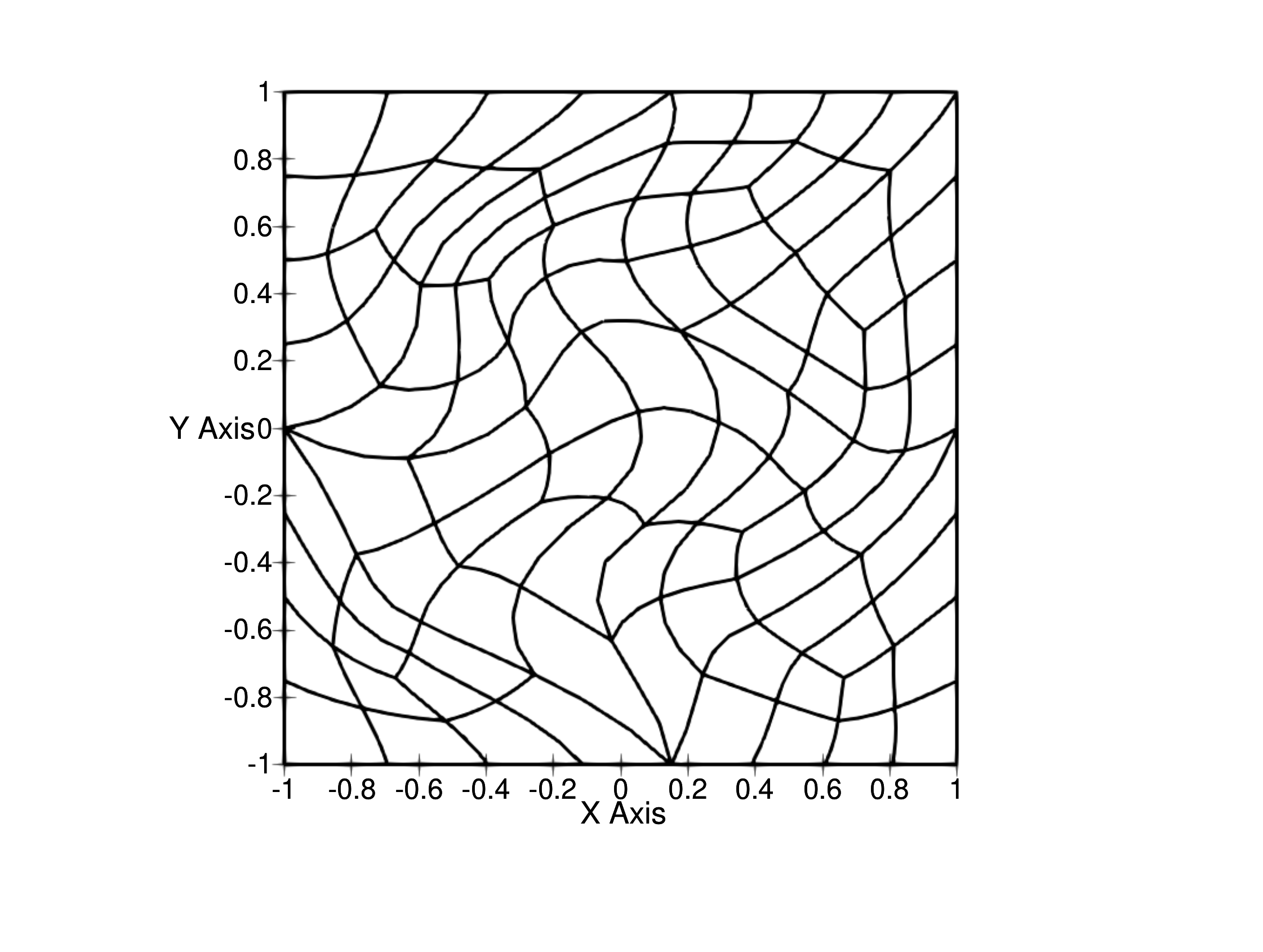}
        \caption{Curvilinear mesh.}
        \label{fig:wb:mesh}
    \end{subfigure}
    \begin{subfigure}{0.32\textwidth}
        \centering
        \includegraphics[width=\linewidth, trim= 4cm 0cm 7cm 0cm , clip]{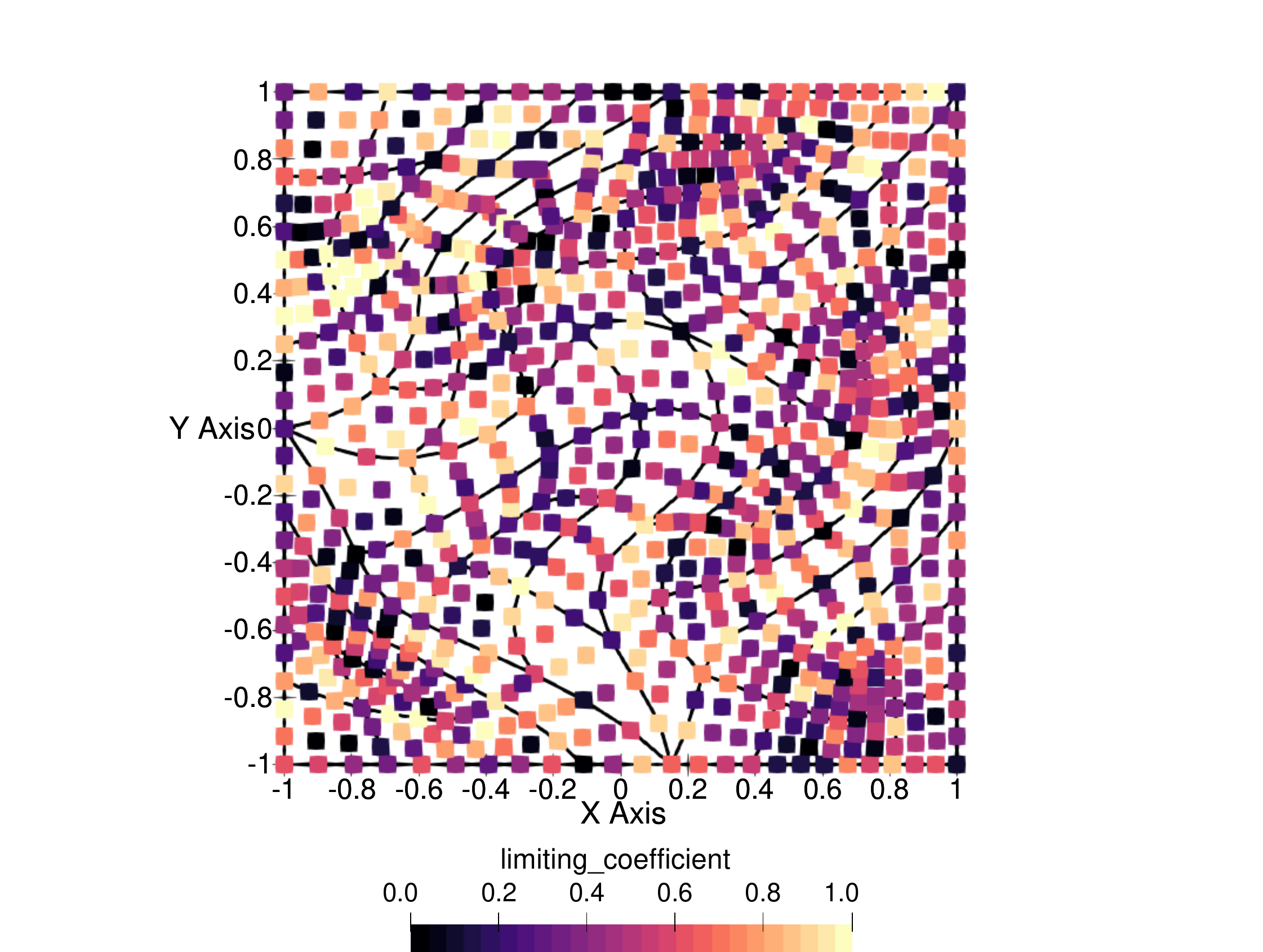}
        \caption{Blending coefficients.}
        \label{fig:wb:alphas}
    \end{subfigure}
    \caption{Curvilinear mesh (a) and random node-wise blending coefficients (b) used for the well-balancedness test.}
    \label{fig:wb:setup}
\end{figure}

We compare the following three numerical schemes:
\begin{enumerate}
    \item[(a)] The well-balanced scheme introduced in Section~\ref{sec:fix-limiting}, which combines the novel flux-differencing formulation with the reformulated fluxes of \citet{ersing2025entropy} (here denoted as \emph{Ersing-jump}). The corresponding two-dimensional fluxes and non-conservative terms are detailed in Appendix~\ref{app:swe_2d}, but are reproduced here for ease of comparison:
\begin{equation}\label{eq:ersing_fluxes_2d}
    \begin{aligned}
   {\state{f}^{1*}} (\state{u}_L, \state{u}_R) = & 
    \begin{pmatrix}
        \avg{h v_1} \\[0.1cm]
        \avg{h v_1} \avg{v_1}\\[0.1cm]
        \avg{h v_1} \avg{v_2}
    \end{pmatrix},
    \quad
     \numnonconsSxi{\tilde{\Jan}} \left(\state{u}_L, \state{u}_R, \left(J\vec{a}^1\right)_L, \left(J\vec{a}^1\right)_R\right) = \begin{pmatrix}
        0 \\[0.1cm] \frac{1}{2}g h_L\left(Ja^1_1\right)_L \\[0.1cm] 
        \frac{1}{2}gh_L\left(Ja^1_2\right)_L
    \end{pmatrix}
    \circ
    \begin{pmatrix}
        0 \\[0.1cm] \jump{h + b} \\[0.1cm] \jump{h + b}
    \end{pmatrix},\\[0.1cm]
    {\state{f}^{2*}} (\state{u}_L, \state{u}_R) = & 
    \begin{pmatrix}
        \avg{h v_2} \\[0.1cm]
        \avg{h v_1}\avg{v_2}\\[0.1cm]
        \avg{h v_2}\avg{v_2}
    \end{pmatrix},
    \quad
    \numnonconsSeta{\tilde{\Jan}} \left(\state{u}_L, \state{u}_R, \left(J\vec{a}^2\right)_L, \left(J\vec{a}^2\right)_R\right) = \begin{pmatrix}
        0 \\[0.1cm] \frac{1}{2}g h_L\left(Ja^2_1\right)_L \\[0.1cm] 
        \frac{1}{2}gh_L\left(Ja^2_2\right)_L
    \end{pmatrix}
    \circ
    \begin{pmatrix}
        0 \\[0.1cm] \jump{h + b} \\[0.1cm] \jump{h + b}
    \end{pmatrix}.
    \end{aligned}
\end{equation}

\item[(b)]  A modified version of the fluxes and non-conservative terms of \citet{wintermeyer2017entropy}, in which the non-conservative contribution is reformulated in local-times-symmetric form (denoted as \emph{Wintermeyer-jump}):
\begin{equation}\label{eq:wintermeyer_local_jump}
    \begin{aligned}
    {\state{f}^{1*}} (\state{u}_L, \state{u}_R) = & 
    \begin{pmatrix}
        \avg{h v_1} \\[0.1cm]
        \avg{h v_1} \avg{v_1}+ \frac{g}{2}h_Lh_R\\[0.1cm]
        \avg{h v_1} \avg{v_2}
    \end{pmatrix},
    \quad
    \numnonconsSxi{\tilde{\Jan}} \left(\state{u}_L, \state{u}_R, \left(J\vec{a}^1\right)_L, \left(J\vec{a}^1\right)_R\right) = \begin{pmatrix}
        0 \\[0.1cm] \frac{1}{2}gh_L \left(Ja^1_1\right)_L \\[0.1cm] \frac{1}{2}gh_L \left(Ja^1_2\right)_L
    \end{pmatrix}
    \circ
    \begin{pmatrix}
        0 \\[0.1cm] \jump{b} \\[0.1cm] \jump{b}
    \end{pmatrix},\\[0.1cm]
    {\state{f}^{2*}} (\state{u}_L, \state{u}_R) = & 
    \begin{pmatrix}
        \avg{h v_2} \\[0.1cm]
        \avg{h v_1} \avg{v_2}\\[0.1cm]
        \avg{h v_2} \avg{v_2} + \frac{g}{2}h_Lh_R
    \end{pmatrix},
    \quad
    \numnonconsSeta{\tilde{\Jan}} \left(\state{u}_L, \state{u}_R, \left(J\vec{a}^2\right)_L, \left(J\vec{a}^2\right)_R\right) = \begin{pmatrix}
        0 \\[0.1cm] \frac{1}{2}gh_L \left(Ja^2_1\right)_L \\[0.1cm] \frac{1}{2}gh_L \left(Ja^2_2\right)_L
    \end{pmatrix}
    \circ
    \begin{pmatrix}
        0 \\[0.1cm] \jump{b} \\[0.1cm] \jump{b}
    \end{pmatrix}.
    \end{aligned}
\end{equation}
Since these fluxes are expressed in local-times-jump form, they are paired with the novel flux-differencing formulation introduced in Section~\ref{sec:fix-limiting}.

\item[(c)] A modified version of the fluxes and non-conservative terms of \citet{wintermeyer2017entropy}, in which the non-conservative contribution is reformulated in local-times-symmetric form (denoted as \emph{Wintermeyer-symmetric}):
\begin{equation}\label{eq:wintermeyer_local_symmetric}
    \begin{aligned}
   {\state{f}^{1*}} (\state{u}_L, \state{u}_R) = & 
    \begin{pmatrix}
        \avg{h v_1} \\[0.1cm]
        \avg{h v_1} \avg{v_1}+ \frac{g}{2}h_Lh_R\\[0.1cm]
        \avg{h v_1} \avg{v_2}
    \end{pmatrix},
    \quad
     \numnonconsSxi{\tilde{\Jan}} \left(\state{u}_L, \state{u}_R, \left(J\vec{a}^1\right)_L, \left(J\vec{a}^1\right)_R\right) = \begin{pmatrix}
        0 \\[0.1cm] gh_L\\[0.1cm] gh_L
    \end{pmatrix}
    \circ
    \begin{pmatrix}
        0 \\[0.1cm] \avg{Ja_1^1}\avg{b} \\[0.1cm] \avg{Ja_2^1}\avg{b}
    \end{pmatrix},\\[0.1cm]
    {\state{f}^{2*}} (\state{u}_L, \state{u}_R) = & 
    \begin{pmatrix}
        \avg{h v_2} \\[0.1cm]
        \avg{h v_1} \avg{v_2}\\[0.1cm]
        \avg{h v_2} \avg{v_2} + \frac{g}{2}h_Lh_R
    \end{pmatrix},
    \quad
    \numnonconsSeta{\tilde{\Jan}} \left(\state{u}_L, \state{u}_R, \left(J\vec{a}^2\right)_L, \left(J\vec{a}^2\right)_R\right) = \begin{pmatrix}
        0 \\[0.1cm] gh_L\\[0.1cm] gh_L
    \end{pmatrix}
    \circ
    \begin{pmatrix}
        0 \\[0.1cm] \avg{Ja_1^2}\avg{b} \\[0.1cm] \avg{Ja_2^2}\avg{b}
    \end{pmatrix}.
    \end{aligned}
\end{equation}
Because these terms are formulated in local-times-symmetric form, they are discretized using the flux-differencing formulation of \citet{rueda2024flux}.
\end{enumerate}

For the surface numerical fluxes and non-conservative interface terms (used both by the DGSEM at element boundaries and by the low-order FV scheme at neighboring nodal interfaces) we include a local Lax--Friedrichs-type dissipation proportional to the jump in entropy variables, as described in Appendix~\ref{app:swe_2d}.

Figure~\ref{fig:wb:error} presents the lake-at-rest error, $H(\vec{x}, t) - H(\vec{x}, 0)$, at time $t=10$ over the domain $\Omega$ for all three configurations. 
The results shown in Figure~\ref{fig:wb:ersing} show that the proposed flux-differencing formulation, when combined with the reformulated fluxes in \eqref{eq:ersing_fluxes_2d}, preserves the lake-at-rest equilibrium to machine precision. 
This confirms that the scheme satisfies the required discrete balance conditions in the hybrid DG/FV setting.

In contrast, combining the novel flux-differencing formulation with the local-times-jump reformulation \eqref{eq:wintermeyer_local_jump} produces spurious oscillations, as illustrated in Figure~\ref{fig:wb:wintermeyer:jump}. 
These errors arise because the flux contributions do not vanish node-wise, as required for Corollary~\ref{cor:nodewise_balance}.

The third configuration, shown in Figure~\ref{fig:wb:wintermeyer:symmetric}, employs the local-times-symmetric flux-differencing formulation of \citet{rueda2024flux} together with the local-times-symmetric reformulation \eqref{eq:wintermeyer_local_symmetric}. 
As expected, this approach also fails to preserve the steady state because it lacks a mechanism to maintain the equilibrium under node-wise blending between the high- and low-order discretizations.

Overall, these results demonstrate that achieving well-balancedness in the hybrid DG/FV framework requires not only a compatible discretization of the fluxes and non-conservative terms, but also a flux-differencing formulation specifically designed to preserve node-wise equilibrium under blending. 
The proposed approach satisfies these requirements and maintains the steady-state to machine precision.

\begin{figure}[]
    \begin{subfigure}{0.32\textwidth}
        \centering
        \includegraphics[width=\linewidth, trim= 4cm 0cm 7cm 0cm , clip]{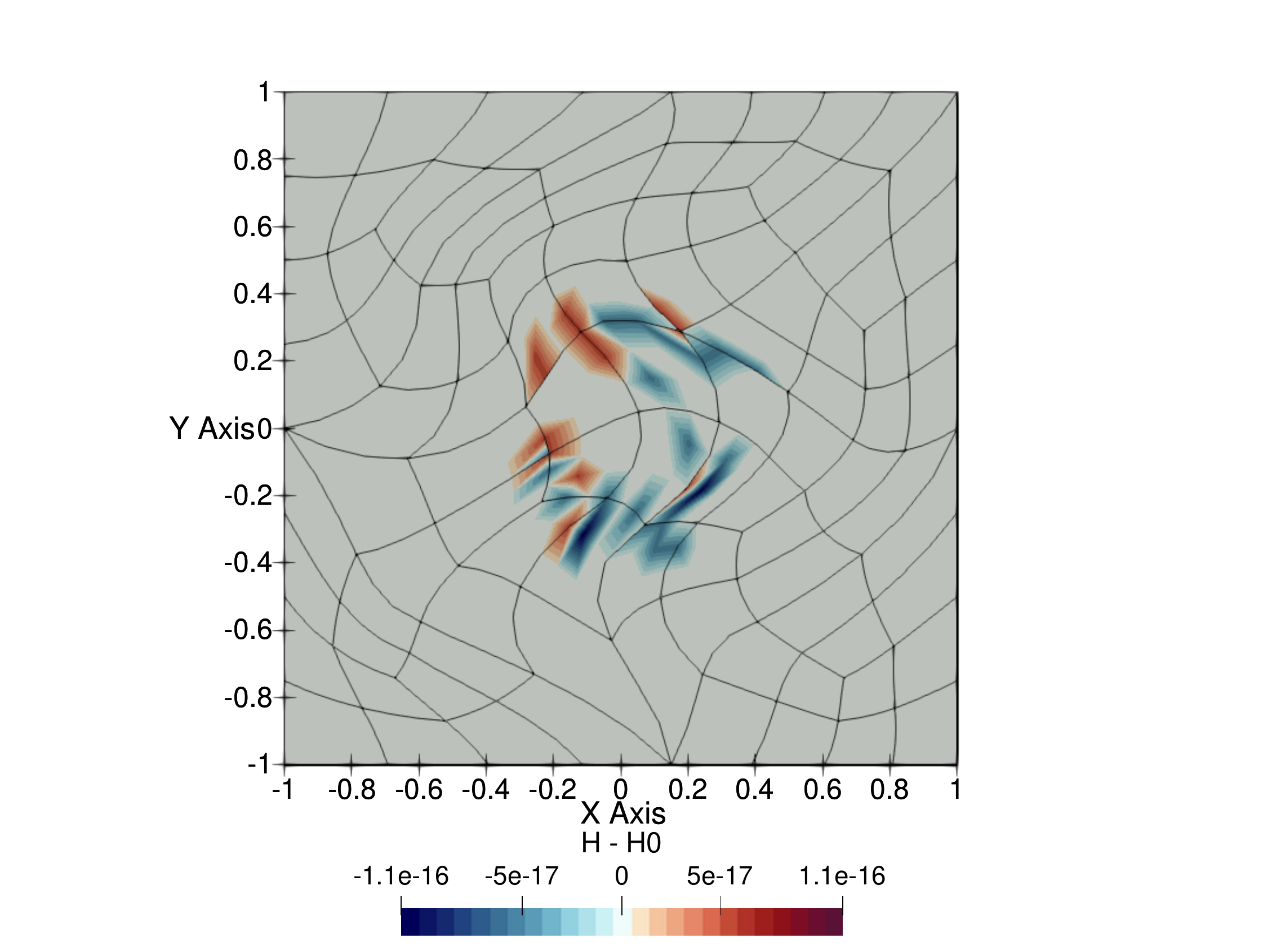}
        \caption{Ersing-jump \eqref{eq:ersing_fluxes_2d}.}
        \label{fig:wb:ersing}
    \end{subfigure}
    \begin{subfigure}{0.32\textwidth}
        \centering
        \includegraphics[width=\linewidth, trim= 4cm 0cm 7cm 0cm , clip]{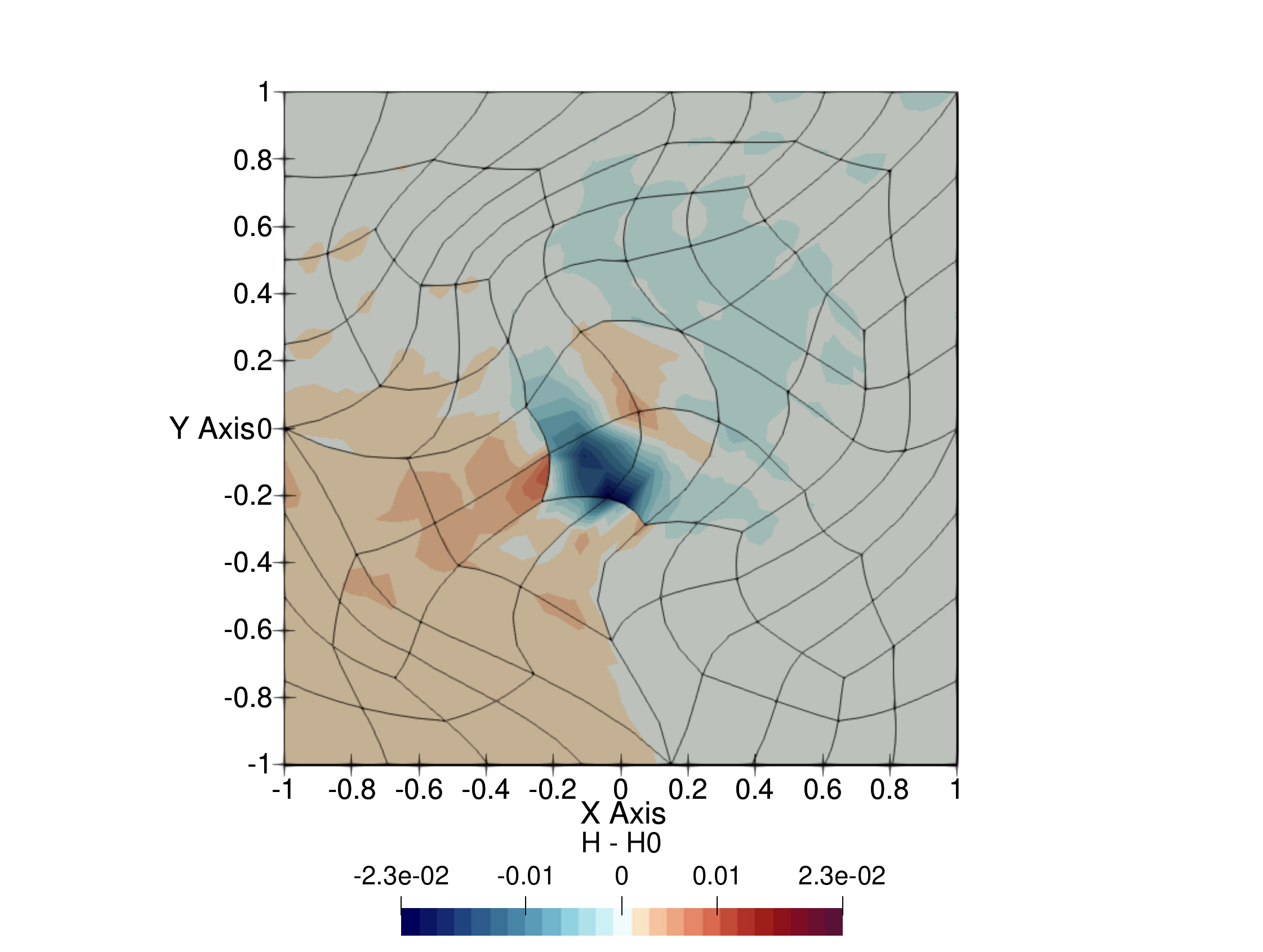}
        \caption{Wintermeyer-jump \eqref{eq:wintermeyer_local_jump}.}
        \label{fig:wb:wintermeyer:jump}
    \end{subfigure}
    \begin{subfigure}{0.32\textwidth}
        \centering
        \includegraphics[width=\linewidth, trim= 4cm 0cm 7cm 0cm , clip]{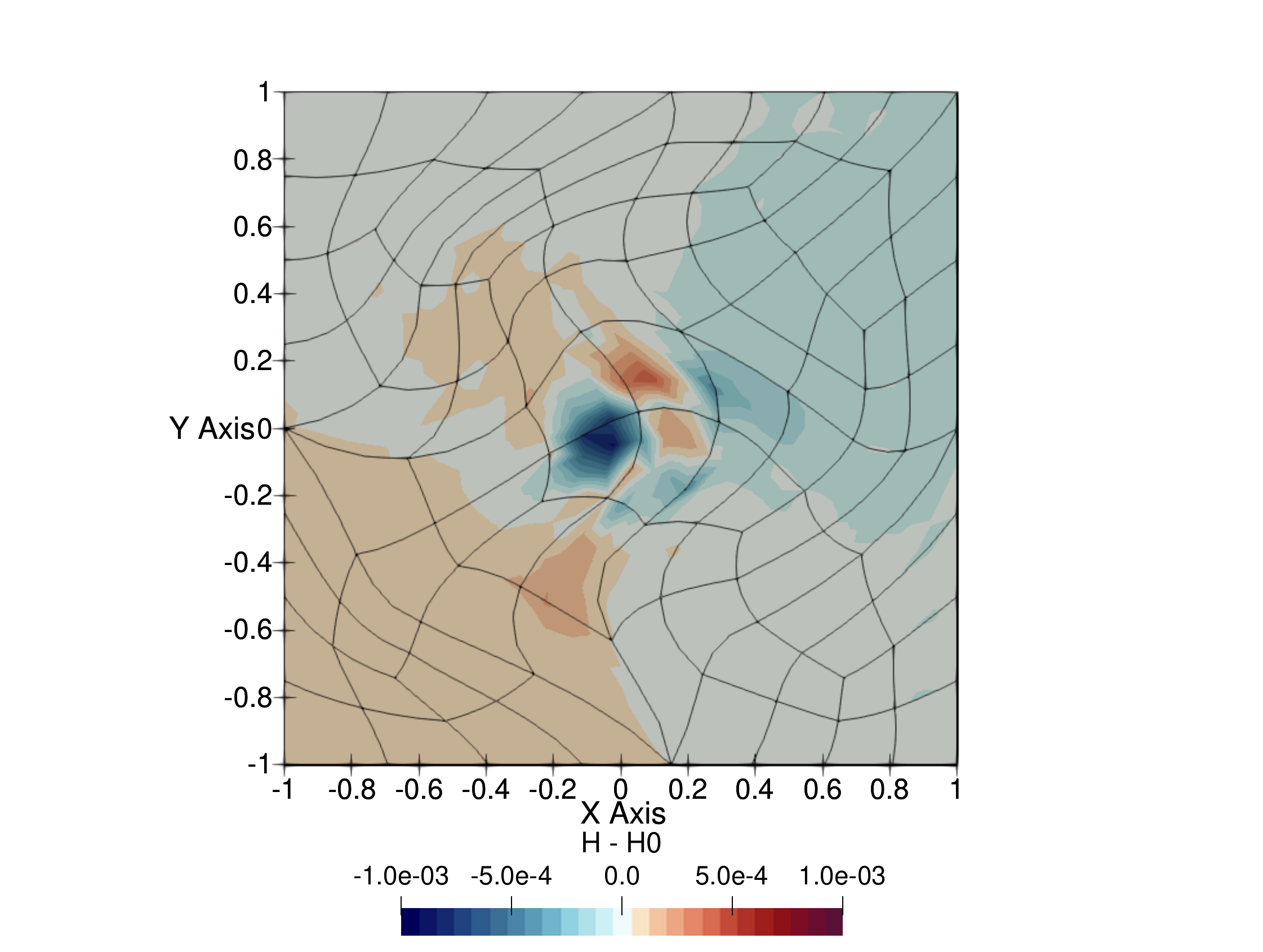}
        \caption{Wintermeyer-symmetric \eqref{eq:wintermeyer_local_symmetric}.}
    \label{fig:wb:wintermeyer:symmetric}
    \end{subfigure}
    \caption{Contour plots showing the lake-at-rest error $H(\vec{x},t) - H(\vec{x},0)$ at time $t=10$ for three different configurations of fluxes and flux-differencing formulas. The color range is adjusted to the solution range in each subplot.}
    \label{fig:wb:error}
\end{figure}

\subsection{Circular dam break} \label{sec:blast}

In this section, we consider a circular dam-break test to highlight the differences among the two flux-differencing formulas that we have considered in this work and to assess the shock-capturing performance of the proposed approach. 
In particular, we compare results obtained with the following flux-differencing formulations, which are consistent with the DGSEM discretization for non-conservative terms expressed in local-times-jump form:
\begin{enumerate}
    \item[(a)] The staggered fluxes proposed in this work: \eqref{eq:leftFlux}--\eqref{eq:rightFlux}.
    
    \item[(b)] The alternative staggered fluxes introduced in Remark~\ref{remark:alternative}, which do not satisfy the pairwise symmetry at interfaces introduced in Remark~\ref{remark:criteria}: \eqref{eq:leftFlux_alternative}--\eqref{eq:rightFlux_alternative}.
\end{enumerate}

We use a coarse Cartesian mesh consisting of $32 \times 32$ elements over the domain $[0,4]^2$, with polynomial degree $N=4$. The initial condition is given by
\begin{equation*}
    H(\vec{x},0) = 
    \begin{cases}
        4,
    & \mathrm{if}~\|\vec{x}\|_2 \le 0.5, \\
        2, & \mathrm{otherwise},
    \end{cases}
    \quad
    \vec{v}_r(\vec{x},0) = 
    \begin{cases}
        0.1882,
    & \mathrm{if}~\|\vec{x}\|_2 \le 0.5, \\
        0, & \mathrm{otherwise},
    \end{cases},
    \quad
    b(\vec{x}) = 0.2 + 0.2 \cos\!\left(\pi (x+y)\right),
\end{equation*}
where $v_r = \vec{v}\cdot\vec{x}/\norm{\vec{x}}_2$ denotes the radial velocity. The simulation uses the two-point fluxes of \citet{ersing2025entropy} together with the modified local-times-jump formulation of the non-conservative terms by \citet{ersing2025entropy} (see Appendix~\ref{app:swe_2d}). 
We set the gravitational acceleration to $g=1$, CFL$=0.4$, and integrate until the final time $t=2$.

For this example, the nodal limiting coefficients are computed to enforce a non-oscillatory behavior of the total water height, $H := h+b$. To this end, we first define local lower and upper bounds for the total water height using the prediction of the low-order FV solver at the next Runge--Kutta stage:
\begin{equation} \label{eq:bounds}
    H^{\min}_{ij} := \min_{k \in \NN (ij)} H^{\FV}_{k},
    \quad
    H^{\max}_{ij} :=
    \max_{k \in \NN (ij)} H^{\FV}_{k},
\end{equation}
where $\NN(ij)$ denotes the low-order stencil associated with node $ij$, namely
\[
\NN(ij) = \{(i-1)j,\,(i+1)j,\,i(j-1),\,i(j+1),\,ij\}.
\]

After computing these bounds, we determine the individual nodal limiting coefficients using a Flux-Corrected Transport (FCT) algorithm based on a Zalesak-type limiter (see, e.g., \cite{zalesak1979fully,kuzmin2010failsafe,kuzmin2012,lohmann2017,RUEDARAMIREZ2022}), ensuring preservation of the bounds in \eqref{eq:bounds}.
The interface blending coefficients between two nodes as the maximum between the nodal coefficients 
\eqref{eq:interface_alpha}.

Figure~\ref{fig:blast_formulas} shows the bottom topography and total water height after 20 time steps for the two staggered flux formulations considered here. 
Among these, the proposed staggered fluxes \eqref{eq:leftFlux}--\eqref{eq:rightFlux} yield the best overall performance, as illustrated in Figure~\ref{fig:blast_newformula}. The alternative formulation, \eqref{eq:leftFlux_alternative}--\eqref{eq:rightFlux_alternative}, which violates the symmetry condition of Remark~\ref{remark:criteria}, produces asymmetric solutions accompanied by over- and undershoots (Figure~\ref{fig:blast_alternative}).

Although the FCT limiting procedure enforces the local bounds defined in \eqref{eq:bounds} and both staggered flux formulations are equivalent to the high-order DGSEM discretization, these numerical experiments suggest that the symmetry condition stated in Remark~\ref{remark:criteria} is essential for obtaining high-quality solutions.

\begin{figure}
    \centering
    \begin{subfigure}[t]{0.465\textwidth}
        \centering
        \includegraphics[width=\textwidth, trim=570 0 570 0, clip]{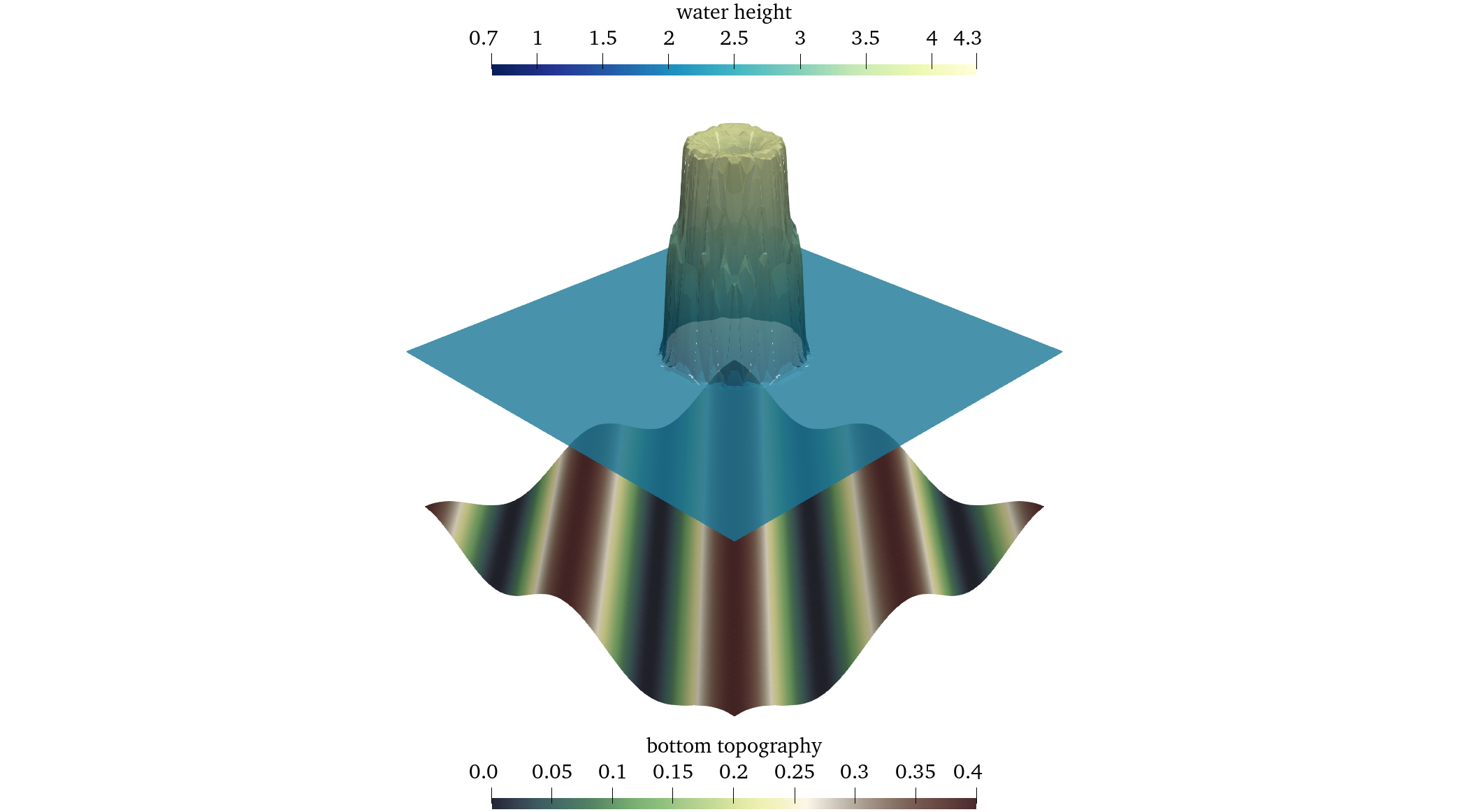}
        \caption{New formula \eqref{eq:leftFlux}--\eqref{eq:rightFlux}.}
        \label{fig:blast_newformula}
    \end{subfigure}
    \qquad
    \begin{subfigure}[t]{0.465\textwidth}
        \centering
        \includegraphics[width=\textwidth, trim=570 0 570 0, clip]{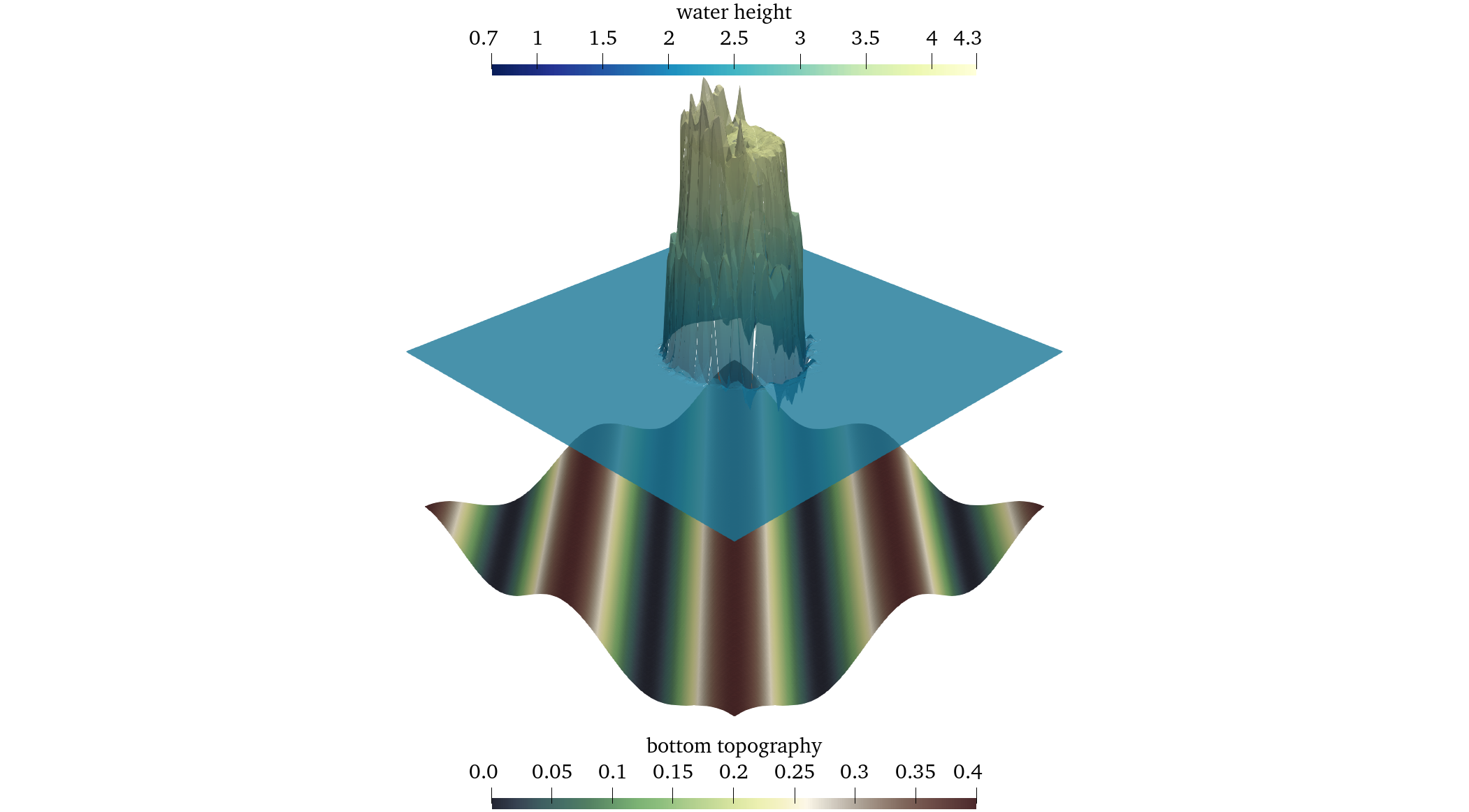}
        \caption{Alternative formula \eqref{eq:leftFlux_alternative}--\eqref{eq:rightFlux_alternative}.}
        \label{fig:blast_alternative}
    \end{subfigure}
    \caption{Three-dimensional visualization of the circular dam break simulation after $20$ time steps for the two different staggered flux formulas considered in this work.}
    \label{fig:blast_formulas}
\end{figure}

\subsection{Dam break past an oblique object}

As a final test, we consider a two-dimensional dam break flow past an isolated obstacle.
This dam break configuration assesses the behavior of the high-order DG approximation in the presence of multiple shock waves, a strong unsteady hydraulic jump, and small-scale eddies that form in wakes zones.
These phenomena are induced due to a dam break from a reservoir over an initially (shallow) wet floodplain that interacts with an oblique building structure. 

This test case exercises the wet/dry transition capabilities of the scheme when the flow interacts with the slanted bottom topography at the top and bottom portions of the channel.
Due to the discontinuous and complex solution structures that emerge, it provides an excellent configuration to compare the resolution capabilities of the novel node-wise limiting compared against the element-wise limiting strategy described in Hennemann et al.~\cite{Hennemann2020}.
To handle the wet/dry transitions we employ the hydrostatic reconstruction technique described in Ersing et al.~\cite{ersing2025entropy}.
As in the previous sections, the node-wise limiting strategy employs a bounds-preserving Zalesak-type limiter \cite{zalesak1979fully,kuzmin2010failsafe,kuzmin2012,lohmann2017,RUEDARAMIREZ2022} to compute the nodal limiting coefficients. 
The interface limiting coefficient is then defined as the maximum of the coefficients associated with the two neighboring nodes, as in \eqref{eq:interface_alpha}. 
For the element-limiting strategy, we adopt a modal energy indicator proposed in \cite{Hennemann2020,Persson2006} calculated from the water height $h$.

The test case domain, shown in Figure~\ref{fig:setup}, is a long channel with a smooth bed, and upstream walls with a gate initially separating a still reservoir of height $0.4$ m with a shallower region of water height of $0.02$ m.
Outflow boundary conditions are used at the right part of the domain. Wall boundary conditions are used at all other domain edges, the dam pylons, and the oblique obstacle.
An experiment of such a 2D dam break flow against an isolated obstacle was carried out by \cite{soares2007experimental} and the results of the experiment at six gauges points, G1--G6, are available for comparison purposes.
The specific locations of each gauge point are given in Table~\ref{tab:2Dposition}.
These gauge points are further illustrated in the domain Figure~\ref{fig:setup} as well as the discretized domain that uses $462$ unstructured quadrilateral elements given in Figure~\ref{fig:mesh}.
\begin{table}[h!]
\centering
    \caption{Position of gauges for 2D dam break flow against an isolated obstacle.}
    \label{tab:2Dposition}
    \begin{tabular}{lcccccc}
    \toprule
    Gauge & G1 & G2 & G3 & G4 & G5 & G6 \\
    \midrule
    $x \, (m)$ & 10.35 & 10.35 & 11.70 & 11.70 & 12.90 & 5.83 \\
    $y \, (m)$ & 2.95 & 1.20 & 2.95 & 1.00 & 2.10 & 2.90 \\
    \bottomrule
\end{tabular}
\end{table}
\begin{figure}
    \centering
    \begin{subfigure}[t]{\textwidth}
        \centering
        \includegraphics[width=\linewidth]{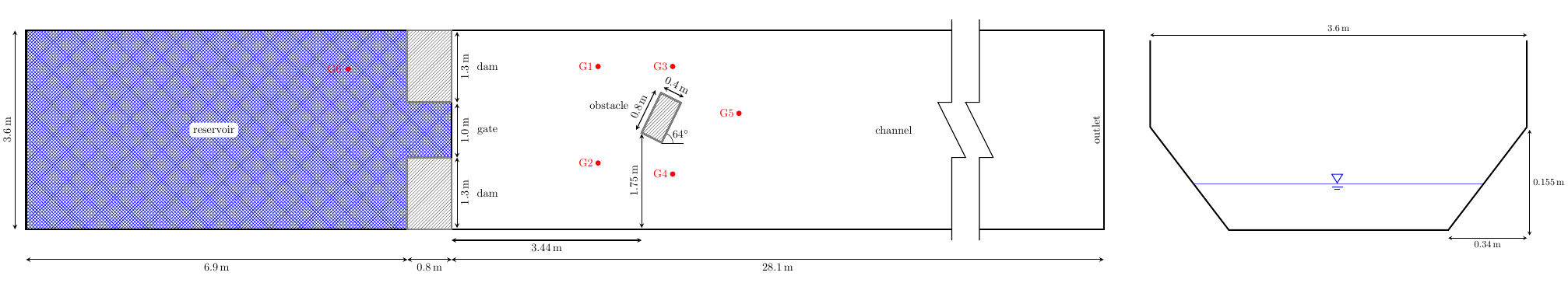}
        \caption{Channel diagram with gauge locations (left)  and channel cross-section (right).}
        \label{fig:setup}
    \end{subfigure}
    \vspace{1em} 
    %
    \begin{subfigure}[t]{\textwidth}
        \centering
        \includegraphics[width=\linewidth]{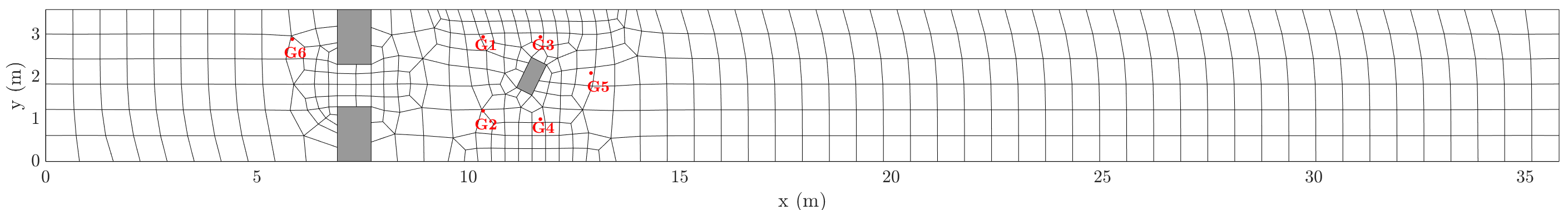}
        \caption{Unstructured quadrilateral mesh with gauge locations.}
        \label{fig:mesh}
    \end{subfigure}
    \caption{(top) Physical domain and cross section for channel flow past an isolated obstacle. (bottom) Decomposition of physical domain into non-overlapping, linear quadrilateral elements.}
\end{figure}

The trapezoidal bottom topography for the channel everywhere except at the dam pylons is represented with the function
\begin{equation}
    \label{eq:channelBottom}
    b(x, y) = \frac{m}{2}\left(|y - 0.34| + |y - 3.26| - 2.92\right),\quad\text{with}\quad m = \frac{0.155}{0.34}.
\end{equation}
To account for bottom friction and allow comparison with experimental data, we introduce an additional source term for the Manning friction with Manning's coefficient set to $n=0.01\,{s}\,{m^{-1/3}}$ \cite{soares2007experimental}.
As the computation contains wet/dry transitions, additional treatment is necessary to ensure robustness and avoid artificially large velocity components that may arise when dividing by water heights near zero.
Therefore, as described in \cite{ersing2025entropy}, we either apply the desingularization formula from \cite{chertock2015well} to adjust the momentum at each node or set the momentum to zero at dry nodes 
\begin{equation}
	hv_{1,2} = 
	\begin{cases}
		\frac{2h^2hv_{1,2}}{h^2 + \max(h^2, \tau_{vel})} & \text{if } h > 5\epsilon_{num},\\[0.1cm]
		0	& \text{otherwise},
	\end{cases}
\end{equation}
where $\epsilon_{num} = 10^{-13}$ is the tolerance for dry states and the threshold $\tau_{vel} = 10^{-8}$ for this problem setup.
This procedure recovers the exact momentum in wet regions and ensures a robust velocity computation and vanishing velocity at dry nodes.
In the presence of partially wet elements, where at least one DG node within an element has a water height below the threshold $10^{-4}$, we adopt an approach from \cite{ersing2025entropy,bonev2018discontinuous} to ensure positivity and well-balancedness for arbitrary wet/dry transition locations.
For this, we modify the limiting coefficient and set $\alpha = 1$ everywhere in said element to fallback to a pure finite volume method in partially dry elements. 
See \cite[Sec. 3.6]{ersing2025entropy} for complete details.

We simulate the dam break past an isolated obstacle up to a final time of $t = 30\ s$ on the discrete domain given by the mesh in Figure~\ref{fig:mesh} using polynomials of degree $N=3$ and $N=5$ in each spatial direction.
We apply the well-balanced node-wise limiting strategy and compare the results against the element-wise limiting strategy provided by~\citet{Hennemann2020}.
For both strategies, we use CFL-based time stepping with $\textrm{CFL}=0.225$.

For this test case, as the gate is swiftly opened, a shock wave propagates and collides with the obstacle forming a reflected
hydraulic jump in the neighborhood of G2, Figure~\ref{fig:setup}.
This collision also generates two shock waves that move in different directions (one towards G3 and one towards G4).
Downstream of the obstacle, a wake zone emerges surrounded by recurrent wave crossings that produce small wake eddies in the region near G5.
This flow behavior is observable in the experimental data from Soares-Fraz{\~a}o and Zech~\cite{soares2007experimental} or in the numerical results of Ginting~\cite{ginting2019central}.

We provide snapshots of the complex solution behavior at $t=10$ in Figure~\ref{fig:N3} for the $N=3$ and in Figure~\ref{fig:N5} for the $N=5$ run, respectively.
For both configurations, we see that the unsteady hydraulic bore that forms in front of the obstacle is captured well.
For the higher resolution $N=5$ simulation we see that the ``X'' flow pattern near the gate is better resolved, especially when using the node-wise limiter.
The results in Figs.~\ref{fig:N3} and \ref{fig:N5} demonstrate that the targeted limiting provided by the node-wise limiter helps preserve complex flow features, whereas the element-wise limiter tends to suppress fine features due to excessive dissipation.
It is interesting to note that the node-wise limiting, although local, does not necessarily activate at the shock location.
Instead, it might activate at the opposite end of an element in order to keep the solution non-oscillatory.
This is especially evident when comparing the $N=3$ variants, where the shock-bore interaction in the neighborhood of G2 is nearly absent for the element-wise simulation.

\begin{figure}
    \begin{subfigure}{\textwidth}
        \centering
        \includegraphics[width=\textwidth, trim=0 0 0 0, clip]{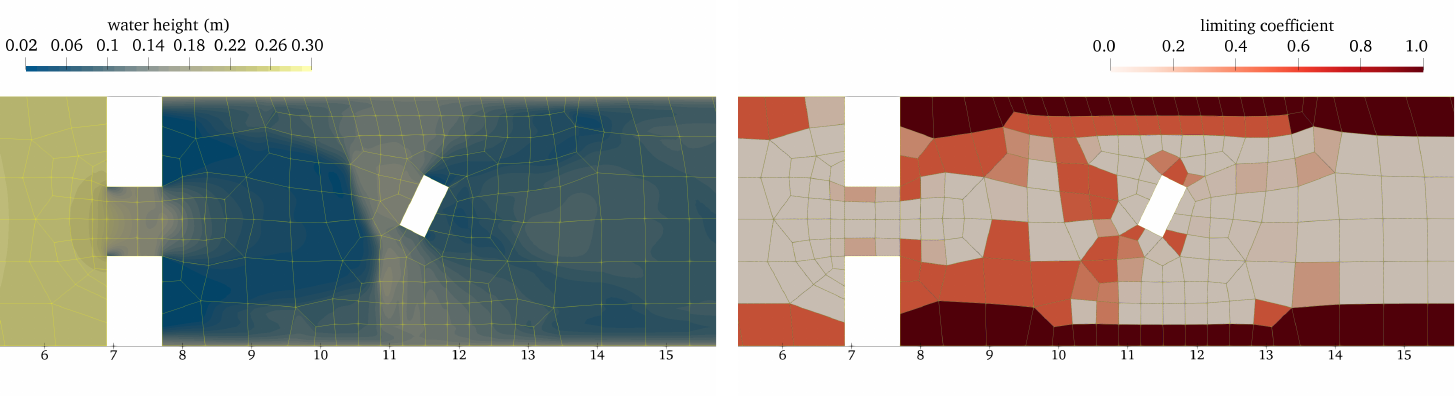}
        \caption{$N=3$, element-wise limiting}
    \end{subfigure}
    
    \begin{subfigure}{\textwidth}
        \centering
        \includegraphics[width=\textwidth, trim=0 0 0 0, clip]{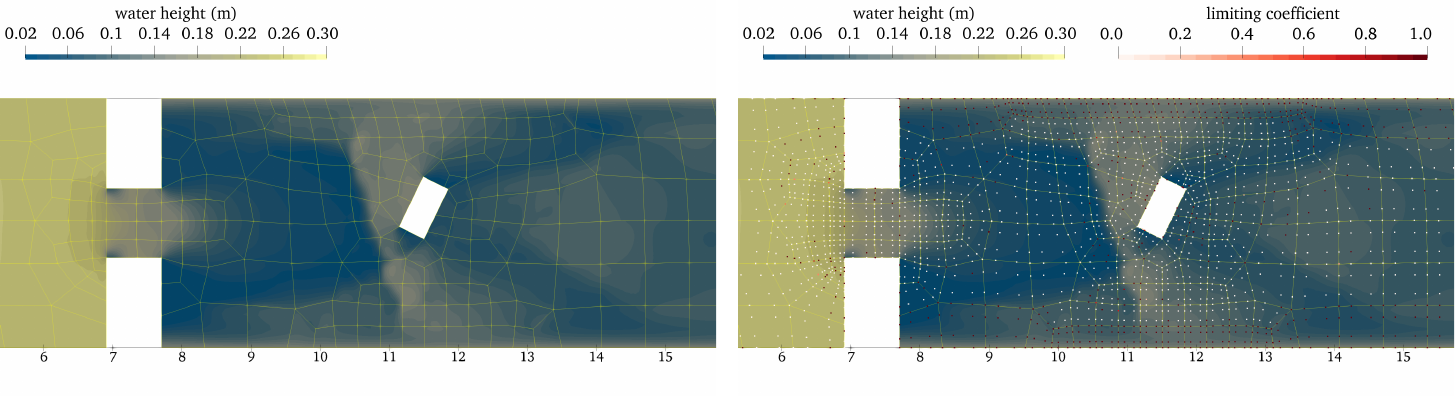}
        \caption{$N=3$, node-wise limiting}
    \end{subfigure}
    \caption{Solution (left) and limiting coefficients (right) for the dam break flow past an obstacle at $t=10$ with polynomial degree $N=3$ in each spatial direction in each element. The node-wise limiting concentrates dissipation near complex flow features.}
    \label{fig:N3}
\end{figure}

\begin{figure}
    \begin{subfigure}{\textwidth}
        \centering
        \includegraphics[width=\textwidth, trim=0 0 0 0, clip]{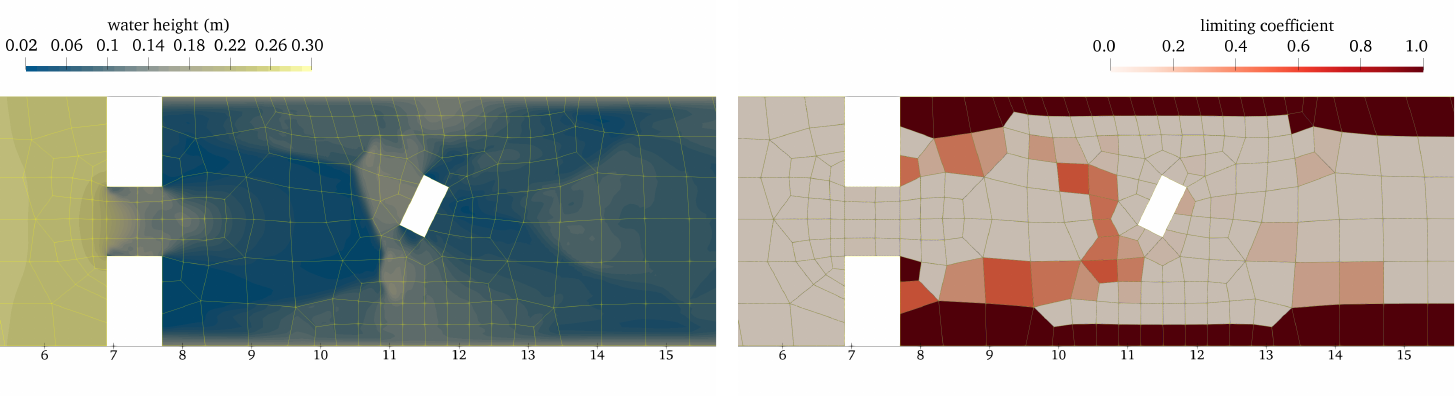}
        \caption{$N=5$, element-wise limiting}
    \end{subfigure}
    
    \begin{subfigure}{\textwidth}
        \centering
        \includegraphics[width=\textwidth, trim=0 0 0 0, clip]{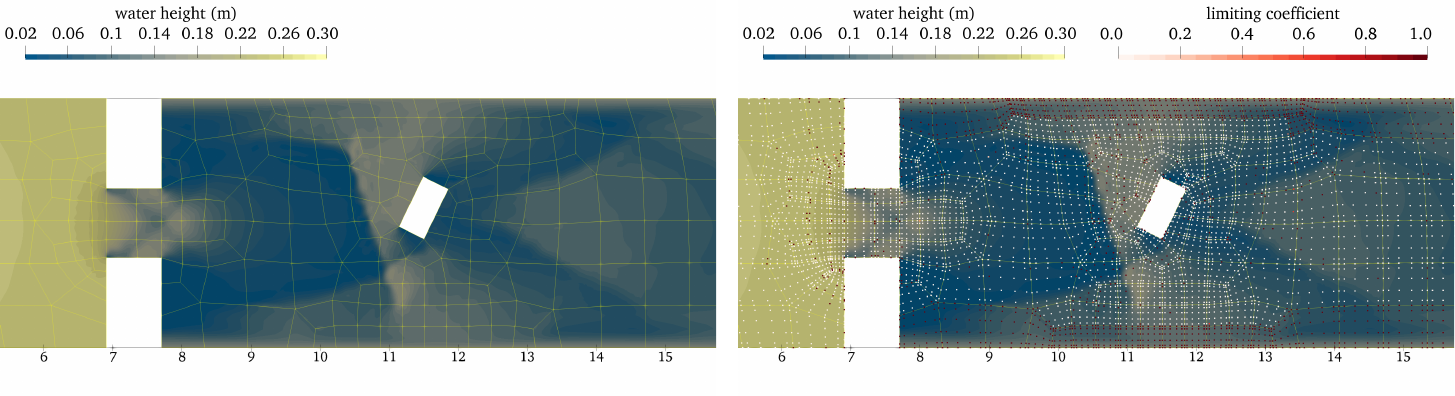}
        \caption{$N=5$, node-wise limiting}
    \end{subfigure}
    \caption{Solution (left) and limiting coefficients (right) for the dam break flow past an obstacle at $t=10$ with polynomial degree $N=5$ in each spatial direction in each element. The node-wise limiting concentrates dissipation near complex flow features.}
    \label{fig:N5}
\end{figure}

To further assess the solution quality, we compare the $N=3$ and $N=5$ element-wise and node-wise limiting results against the experimental data from \cite{soares2007experimental}.
We present the results for the water height at G2, G4, G5, and G6 in Figure~\ref{fig:gauges}.
All the results compare well with those presented in the literature for this test case \cite{cea2015simple,hou20132d,ginting2019central,ayog2021second}.
For all gauges we see that the node-wise limiter better captures the flow behavior compared to the element-wise limiting solution.
We consider G2, G4 and G5 located around the obstacle where complex energetic flow features occur.
Particularly for G2, the results in Figure~\ref{fig:gauges} capture the first incoming wave properly at around $1.5\ s$ for both resolutions.
However, at the transition time between $t=13\ s$ to $t=19\ s$, as the unsteady bore propagates to the left, the node-wise limiting captures the flow behavior better at either resolution.
At G4, all runs detect the maximum bore at $2\ s$ accurately; however, between $2$ and $5\ s$ it predicts the water height slightly higher.
The behavior afterward is similar for both resolutions.
The water height recorded at point G5 differs slightly between the resolutions.
The higher resolution $N=5$ is slightly better at capturing the eddy behavior in the wake region behind the obstacle.
For the gauge G6 within the reservoir, both resolutions accurately capture the ``draining'' of the water.

\begin{figure}[]
    \begin{subfigure}{0.495\textwidth}
        \centering
        \includegraphics[width=\linewidth, trim= 0 0 0 0, clip]{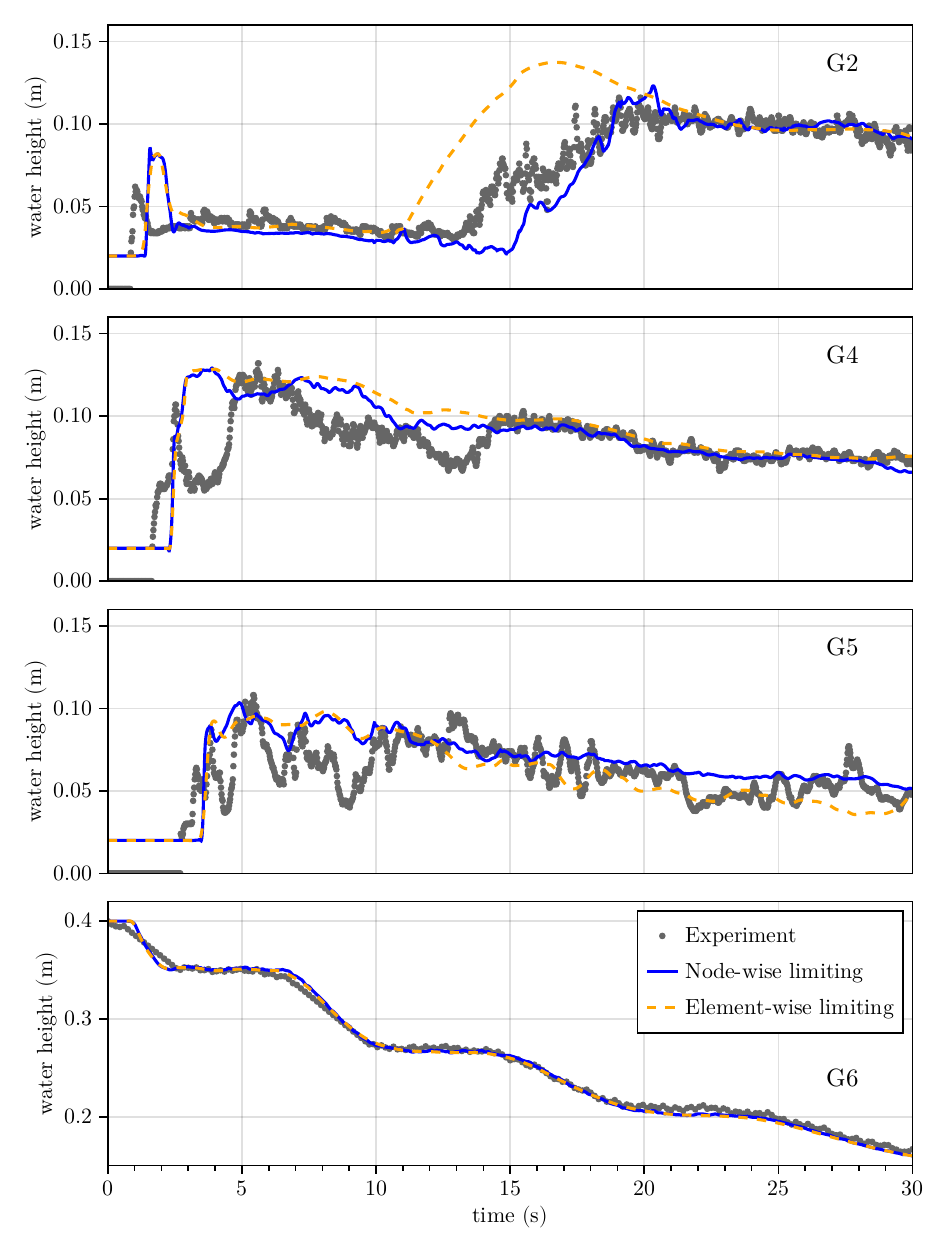}
        \caption{$N=3$}
    \end{subfigure}
    \begin{subfigure}{0.495\textwidth}
        \centering
        \includegraphics[width=\linewidth, trim= 0 0 0 0, clip]{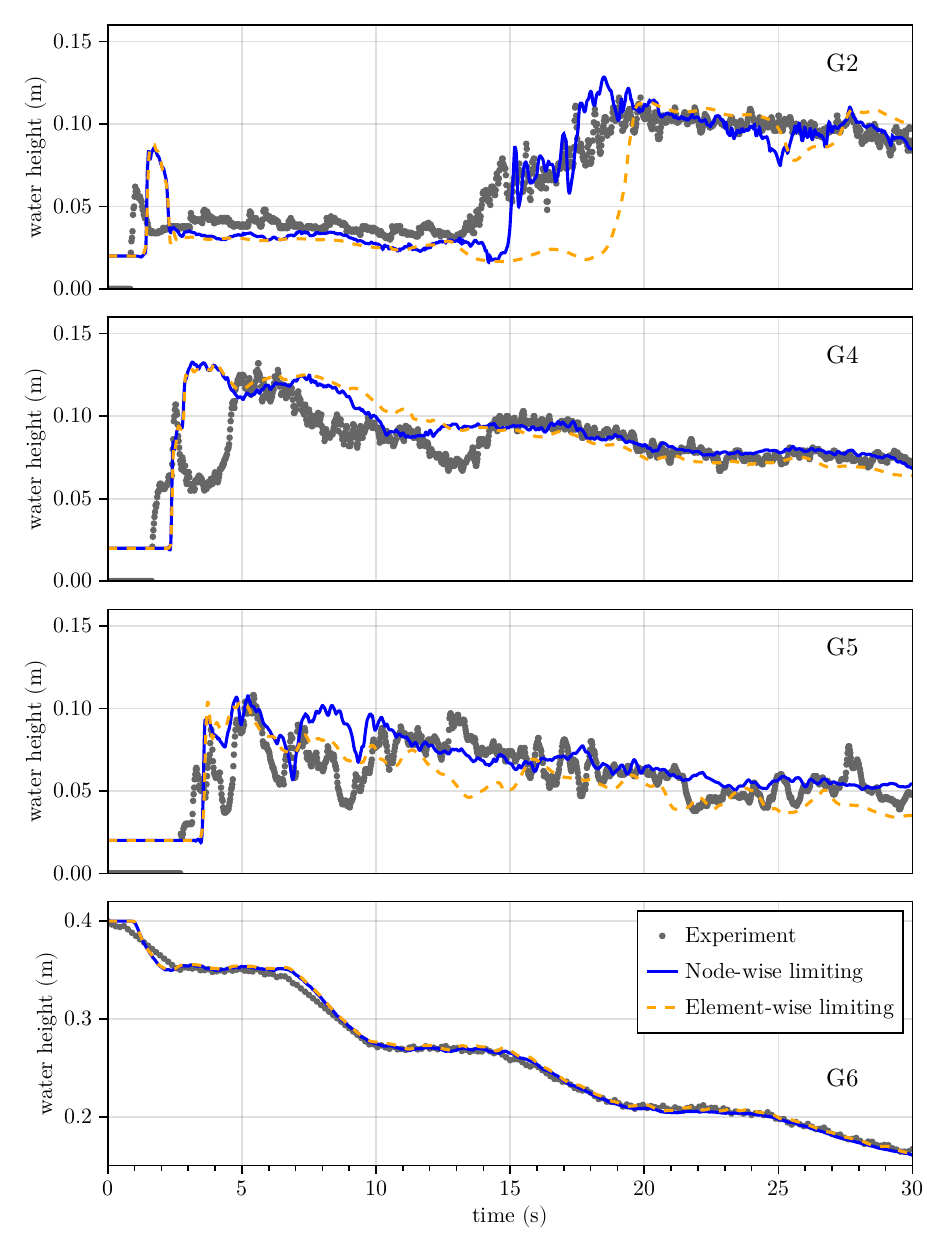}
        \caption{$N=5$}
    \end{subfigure}
    \caption{Water height histories at gauges G2, G4, G5, and G6.}
    \label{fig:gauges}
\end{figure}

These results serve to demonstrate that the moderately resolved simulations with $7,932$ DOFs per equation ($N=3$) or $16,632$ DOFs per equation ($N=5$) yield acceptable results for this dam break past an obstacle test case compared against experimental~\cite{soares2007experimental} or high resolution $2^{\textrm{nd}}$ order benchmark data~\cite{hou20132d,ginting2019central}.
The results above further demonstrate that the targeted dissipation offered by the novel, well-balanced node-wise limiting strategy provides superior results compared to the element-wise strategy, for this complex test case with hydraulic bores, shocks, and wet/dry interactions.

\section{Conclusions} \label{sec:Conclusions}

In this work, we have developed a novel hybrid DG/FV framework for non-conservative systems that enables robust node-wise subcell limiting while exactly preserving well-balanced steady states. 
The central challenge addressed in this paper is that standard flux-differencing formulations generally destroy well-balancedness under nodal blending, even when both the underlying DG and FV discretizations are individually well-balanced.

To overcome this limitation, we propose a methodology consisting of two key steps.
First, we introduce a suitable reformulation of the non-conservative system in which the source term is expressed in terms of the gradient or divergence of a quantity that vanishes at equilibrium. 
This reformulation enables the construction of compatible two-point fluxes based on the jump of that quantity. 
Second, we develop a new flux-differencing formulation in which the non-conservative contributions are written in a local-times-jump form. 
This structure ensures that each staggered DG flux vanishes individually at equilibrium, thereby guaranteeing well-balancedness at the node level.

As a representative case study, we consider the shallow water equations. 
In this setting, the proposed reformulation expresses the source term in terms of the gradient of the total water height. 
This enables the construction of two-point fluxes whose jump contributions vanish for equilibrium states. 
Combined with the novel flux-differencing formulation, the resulting hybrid DG/FV scheme is capable of capturing shocks while remaining both bounds-preserving and exactly well-balanced under arbitrary node-wise selections of the limiting coefficients.

The proposed method retains the high-order accuracy of the underlying DGSEM discretization, as confirmed by convergence studies, while achieving exact preservation of lake-at-rest steady states under node-wise limiting. 
Numerical experiments demonstrate that the scheme significantly improves robustness and stability compared to existing approaches, particularly in challenging scenarios involving wet/dry fronts, shocks, and complex flow interactions. 
Furthermore, the node-wise limiting strategy introduces more targeted dissipation, leading to improved resolution of fine-scale flow features when compared to traditional element-wise limiting techniques.

Although the present work focuses on the shallow water equations, the proposed framework is applicable to a broader class of balance laws, provided that the non-conservative terms admit a compatible local-times-jump representation.

\section*{Acknowledgments}

Andrés M. Rueda-Ramírez gratefully acknowledges funding from the Spanish Ministry of Science, Innovation, and Universities through the ``Beatriz Galindo'' grant (BG23-00062).
This project has received funding from the European Research Council (ERC) under the European Union's Horizon Europe research and innovation programme (grant agreement No. 101167322 - TRANSDIFFUSE).
Gregor J. Gassner and Andrés M. Rueda-Ramírez acknowledge funding through the German Federal Ministry for Education and Research (BMBF) project ``ICON-DG'' (01LK2315B) of the ``WarmWorld Smarter'' program.


\printcredits

\bibliographystyle{model1-num-names}

\bibliography{Biblio.bib}

\begin{thebibliography}{63}
\expandafter\ifx\csname natexlab\endcsname\relax\def\natexlab#1{#1}\fi
\providecommand{\bibinfo}[2]{#2}
\ifx\xfnm\relax \def\xfnm[#1]{\unskip,\space#1}\fi
\bibitem[{Wang et~al.(2013)Wang, Fidkowski, Abgrall, Bassi, Caraeni, Cary, Deconinck, Hartmann, Hillewaert, Huynh, Kroll, May, Persson, van Leer, and Visbal}]{Wang2013High}
\bibinfo{author}{Z.~J. Wang}, \bibinfo{author}{K.~Fidkowski}, \bibinfo{author}{R.~Abgrall}, \bibinfo{author}{F.~Bassi}, \bibinfo{author}{D.~Caraeni}, \bibinfo{author}{A.~Cary}, \bibinfo{author}{H.~Deconinck}, \bibinfo{author}{R.~Hartmann}, \bibinfo{author}{K.~Hillewaert}, \bibinfo{author}{H.~T. Huynh}, \bibinfo{author}{N.~Kroll}, \bibinfo{author}{G.~May}, \bibinfo{author}{P.-O. Persson}, \bibinfo{author}{B.~van Leer}, \bibinfo{author}{M.~Visbal},
\newblock \bibinfo{title}{{High-order CFD methods: current status and perspective}},
\newblock \bibinfo{journal}{International Journal for Numerical Methods in Fluids} \bibinfo{volume}{72} (\bibinfo{year}{2013}) \bibinfo{pages}{811--845}.
\bibitem[{Cockburn et~al.(2000)Cockburn, Karniadakis, and Shu}]{Cockburn2000}
\bibinfo{author}{B.~Cockburn}, \bibinfo{author}{G.~E. Karniadakis}, \bibinfo{author}{C.-W. Shu},
\newblock \bibinfo{title}{{The Development of Discontinuous Galerkin Methods}},
\newblock \bibinfo{journal}{Discontinuous Galerkin Methods} \bibinfo{volume}{11} (\bibinfo{year}{2000}) \bibinfo{pages}{3--50}.
\bibitem[{Hindenlang et~al.(2012)Hindenlang, Gassner, Altmann, Beck, Staudenmaier, and Munz}]{Hindenlang2012}
\bibinfo{author}{F.~Hindenlang}, \bibinfo{author}{G.~J. Gassner}, \bibinfo{author}{C.~Altmann}, \bibinfo{author}{A.~Beck}, \bibinfo{author}{M.~Staudenmaier}, \bibinfo{author}{C.~D. Munz},
\newblock \bibinfo{title}{{Explicit discontinuous Galerkin methods for unsteady problems}},
\newblock \bibinfo{journal}{Computers and Fluids} \bibinfo{volume}{61} (\bibinfo{year}{2012}) \bibinfo{pages}{86--93}.
\bibitem[{Ranocha et~al.(2021)Ranocha, Schlottke-Lakemper, Chan, Rueda-Ram{\'\i}rez, Winters, Hindenlang, and Gassner}]{ranocha2021efficient}
\bibinfo{author}{H.~Ranocha}, \bibinfo{author}{M.~Schlottke-Lakemper}, \bibinfo{author}{J.~Chan}, \bibinfo{author}{A.~M. Rueda-Ram{\'\i}rez}, \bibinfo{author}{A.~R. Winters}, \bibinfo{author}{F.~Hindenlang}, \bibinfo{author}{G.~J. Gassner},
\newblock \bibinfo{title}{Efficient implementation of modern entropy stable and kinetic energy preserving discontinuous {G}alerkin methods for conservation laws},
\newblock \bibinfo{journal}{arXiv preprint arXiv:2112.10517}  (\bibinfo{year}{2021}).
\bibitem[{Fisher and Carpenter(2013)}]{Fisher2013a}
\bibinfo{author}{T.~C. Fisher}, \bibinfo{author}{M.~H. Carpenter},
\newblock \bibinfo{title}{{High-order entropy stable finite difference schemes for nonlinear conservation laws: Finite domains}},
\newblock \bibinfo{journal}{Journal of Computational Physics} \bibinfo{volume}{252} (\bibinfo{year}{2013}) \bibinfo{pages}{518--557}.
\bibitem[{Carpenter et~al.(2014)Carpenter, Fisher, Nielsen, and Frankel}]{Carpenter2014}
\bibinfo{author}{M.~H. Carpenter}, \bibinfo{author}{T.~C. Fisher}, \bibinfo{author}{E.~J. Nielsen}, \bibinfo{author}{S.~H. Frankel},
\newblock \bibinfo{title}{{Entropy stable spectral collocation schemes for the Navier-Stokes Equations: Discontinuous interfaces}},
\newblock \bibinfo{journal}{SIAM Journal on Scientific Computing} \bibinfo{volume}{36} (\bibinfo{year}{2014}) \bibinfo{pages}{B835--B867}.
\bibitem[{Gassner(2013)}]{Gassner2013}
\bibinfo{author}{G.~J. Gassner},
\newblock \bibinfo{title}{{A Skew-Symmetric Discontinuous Galerkin Spectral Element Discretization and Its Relation to SBP-SAT Finite Difference Methods}},
\newblock \bibinfo{journal}{SIAM Journal on Scientific Computing} \bibinfo{volume}{35} (\bibinfo{year}{2013}) \bibinfo{pages}{A1233--A1253}.
\bibitem[{Gassner et~al.(2016)Gassner, Winters, and Kopriva}]{Gassner2016}
\bibinfo{author}{G.~J. Gassner}, \bibinfo{author}{A.~R. Winters}, \bibinfo{author}{D.~A. Kopriva},
\newblock \bibinfo{title}{{Split form nodal discontinuous Galerkin schemes with summation-by-parts property for the compressible Euler equations}},
\newblock \bibinfo{journal}{Journal of Computational Physics} \bibinfo{volume}{327} (\bibinfo{year}{2016}) \bibinfo{pages}{39--66}.
\bibitem[{Bohm et~al.(2018)Bohm, Winters, Gassner, Derigs, Hindenlang, and Saur}]{Bohm2018}
\bibinfo{author}{M.~Bohm}, \bibinfo{author}{A.~R. Winters}, \bibinfo{author}{G.~J. Gassner}, \bibinfo{author}{D.~Derigs}, \bibinfo{author}{F.~Hindenlang}, \bibinfo{author}{J.~Saur},
\newblock \bibinfo{title}{{An entropy stable nodal discontinuous Galerkin method for the resistive MHD equations. Part I: Theory and numerical verification}},
\newblock \bibinfo{journal}{Journal of Computational Physics} \bibinfo{volume}{1} (\bibinfo{year}{2018}) \bibinfo{pages}{1--35}.
\bibitem[{Wintermeyer et~al.(2017)Wintermeyer, Winters, Gassner, and Kopriva}]{wintermeyer2017entropy}
\bibinfo{author}{N.~Wintermeyer}, \bibinfo{author}{A.~R. Winters}, \bibinfo{author}{G.~J. Gassner}, \bibinfo{author}{D.~A. Kopriva},
\newblock \bibinfo{title}{An entropy stable nodal discontinuous {G}alerkin method for the two dimensional shallow water equations on unstructured curvilinear meshes with discontinuous bathymetry},
\newblock \bibinfo{journal}{Journal of Computational Physics} \bibinfo{volume}{340} (\bibinfo{year}{2017}) \bibinfo{pages}{200--242}.
\bibitem[{Rueda-Ram{\'\i}rez et~al.(2023)Rueda-Ram{\'\i}rez, Hindenlang, Chan, and Gassner}]{rueda2023entropy}
\bibinfo{author}{A.~M. Rueda-Ram{\'\i}rez}, \bibinfo{author}{F.~J. Hindenlang}, \bibinfo{author}{J.~Chan}, \bibinfo{author}{G.~J. Gassner},
\newblock \bibinfo{title}{Entropy-stable {G}auss collocation methods for ideal magneto-hydrodynamics},
\newblock \bibinfo{journal}{Journal of Computational Physics} \bibinfo{volume}{475} (\bibinfo{year}{2023}) \bibinfo{pages}{111851}.
\bibitem[{Rueda-Ram{\'\i}rez et~al.(2025)Rueda-Ram{\'\i}rez, Sikstel, and Gassner}]{rueda2025entropy}
\bibinfo{author}{A.~M. Rueda-Ram{\'\i}rez}, \bibinfo{author}{A.~Sikstel}, \bibinfo{author}{G.~J. Gassner},
\newblock \bibinfo{title}{An entropy-stable discontinuous {G}alerkin discretization of the ideal multi-ion magnetohydrodynamics system},
\newblock \bibinfo{journal}{Journal of Computational Physics} \bibinfo{volume}{523} (\bibinfo{year}{2025}) \bibinfo{pages}{113655}.
\bibitem[{Coquel et~al.(2021)Coquel, Marmignon, Rai, and Renac}]{coquel2021entropy}
\bibinfo{author}{F.~Coquel}, \bibinfo{author}{C.~Marmignon}, \bibinfo{author}{P.~Rai}, \bibinfo{author}{F.~Renac},
\newblock \bibinfo{title}{An entropy stable high-order discontinuous {G}alerkin spectral element method for the {B}aer-{N}unziato two-phase flow model},
\newblock \bibinfo{journal}{Journal of Computational Physics} \bibinfo{volume}{431} (\bibinfo{year}{2021}) \bibinfo{pages}{110135}.
\bibitem[{Waruszewski et~al.(2022)Waruszewski, Kozdon, Wilcox, Gibson, and Giraldo}]{waruszewski2022entropy}
\bibinfo{author}{M.~Waruszewski}, \bibinfo{author}{J.~E. Kozdon}, \bibinfo{author}{L.~C. Wilcox}, \bibinfo{author}{T.~H. Gibson}, \bibinfo{author}{F.~X. Giraldo},
\newblock \bibinfo{title}{Entropy stable discontinuous {G}alerkin methods for balance laws in non-conservative form: {A}pplications to the {E}uler equations with gravity},
\newblock \bibinfo{journal}{Journal of Computational Physics} \bibinfo{volume}{468} (\bibinfo{year}{2022}) \bibinfo{pages}{111507}.
\bibitem[{Hennemann et~al.(2020)Hennemann, Rueda-Ram{\'{i}}rez, Hindenlang, and Gassner}]{Hennemann2020}
\bibinfo{author}{S.~Hennemann}, \bibinfo{author}{A.~M. Rueda-Ram{\'{i}}rez}, \bibinfo{author}{F.~J. Hindenlang}, \bibinfo{author}{G.~J. Gassner},
\newblock \bibinfo{title}{{A provably entropy stable subcell shock capturing approach for high order split form DG for the compressible Euler equations}},
\newblock \bibinfo{journal}{Journal of Computational Physics}  (\bibinfo{year}{2020}) \bibinfo{pages}{109935}.
\bibitem[{Rueda-Ram{\'{i}}rez et~al.(2021)Rueda-Ram{\'{i}}rez, Hennemann, Hindenlang, Winters, and Gassner}]{Rueda-Ramirez2020}
\bibinfo{author}{A.~M. Rueda-Ram{\'{i}}rez}, \bibinfo{author}{S.~Hennemann}, \bibinfo{author}{F.~J. Hindenlang}, \bibinfo{author}{A.~R. Winters}, \bibinfo{author}{G.~J. Gassner}, \bibinfo{title}{{An entropy stable nodal discontinuous Galerkin method for the resistive MHD equations. Part II: Subcell finite volume shock capturing}}, volume \bibinfo{volume}{444}, \bibinfo{year}{2021}.
\bibitem[{Rueda-Ramírez et~al.(2022)Rueda-Ramírez, Pazner, and Gassner}]{RUEDARAMIREZ2022}
\bibinfo{author}{A.~M. Rueda-Ramírez}, \bibinfo{author}{W.~Pazner}, \bibinfo{author}{G.~J. Gassner},
\newblock \bibinfo{title}{Subcell limiting strategies for discontinuous {G}alerkin spectral element methods},
\newblock \bibinfo{journal}{Computers \& Fluids} \bibinfo{volume}{247} (\bibinfo{year}{2022}) \bibinfo{pages}{105627}.
\bibitem[{Rueda-Ram{\'\i}rez et~al.(2023)Rueda-Ram{\'\i}rez, Bolm, Kuzmin, and Gassner}]{rueda2023monolithic}
\bibinfo{author}{A.~M. Rueda-Ram{\'\i}rez}, \bibinfo{author}{B.~Bolm}, \bibinfo{author}{D.~Kuzmin}, \bibinfo{author}{G.~J. Gassner},
\newblock \bibinfo{title}{Monolithic convex limiting for {L}egendre-{G}auss-{L}obatto discontinuous {G}alerkin spectral element methods},
\newblock \bibinfo{journal}{arXiv preprint arXiv:2303.00374}  (\bibinfo{year}{2023}).
\bibitem[{Mateo-Gab{\'\i}n et~al.(2023)Mateo-Gab{\'\i}n, Rueda-Ram{\'\i}rez, Valero, and Rubio}]{mateo2023flux}
\bibinfo{author}{A.~Mateo-Gab{\'\i}n}, \bibinfo{author}{A.~M. Rueda-Ram{\'\i}rez}, \bibinfo{author}{E.~Valero}, \bibinfo{author}{G.~Rubio},
\newblock \bibinfo{title}{A flux-differencing formulation with {G}auss nodes},
\newblock \bibinfo{journal}{Journal of Computational Physics} \bibinfo{volume}{489} (\bibinfo{year}{2023}) \bibinfo{pages}{112298}.
\bibitem[{Kuzmin(2020)}]{kuzmin2020monolithic}
\bibinfo{author}{D.~Kuzmin},
\newblock \bibinfo{title}{Monolithic convex limiting for continuous finite element discretizations of hyperbolic conservation laws},
\newblock \bibinfo{journal}{Computer Methods in Applied Mechanics and Engineering} \bibinfo{volume}{361} (\bibinfo{year}{2020}) \bibinfo{pages}{112804}.
\bibitem[{Guermond et~al.(2019)Guermond, Popov, and Tomas}]{guermond2019invariant}
\bibinfo{author}{J.-L. Guermond}, \bibinfo{author}{B.~Popov}, \bibinfo{author}{I.~Tomas},
\newblock \bibinfo{title}{Invariant domain preserving discretization-independent schemes and convex limiting for hyperbolic systems},
\newblock \bibinfo{journal}{Computer Methods in Applied Mechanics and Engineering} \bibinfo{volume}{347} (\bibinfo{year}{2019}) \bibinfo{pages}{143--175}.
\bibitem[{Pazner(2020)}]{Pazner2020}
\bibinfo{author}{W.~Pazner},
\newblock \bibinfo{title}{{Sparse Invariant Domain Preserving Discontinuous Galerkin Methods With Subcell Convex Limiting}},
\newblock \bibinfo{journal}{arXiv}  (\bibinfo{year}{2020}).
\bibitem[{Lin and Chan(2024)}]{lin2024high}
\bibinfo{author}{Y.~Lin}, \bibinfo{author}{J.~Chan},
\newblock \bibinfo{title}{High order entropy stable discontinuous {G}alerkin spectral element methods through subcell limiting},
\newblock \bibinfo{journal}{Journal of Computational Physics} \bibinfo{volume}{498} (\bibinfo{year}{2024}) \bibinfo{pages}{112677}.
\bibitem[{Ranocha(2017)}]{ranocha2017shallow}
\bibinfo{author}{H.~Ranocha},
\newblock \bibinfo{title}{Shallow water equations: {S}plit-form, entropy stable, well-balanced, and positivity preserving numerical methods},
\newblock \bibinfo{journal}{GEM-International Journal on Geomathematics} \bibinfo{volume}{8} (\bibinfo{year}{2017}) \bibinfo{pages}{85--133}.
\bibitem[{Winters et~al.(2025)Winters, Ersing, Ranocha, and Schlottke-Lakemper}]{winters2025trixi}
\bibinfo{author}{A.~R. Winters}, \bibinfo{author}{P.~Ersing}, \bibinfo{author}{H.~Ranocha}, \bibinfo{author}{M.~Schlottke-Lakemper}, \bibinfo{title}{{TrixiShallowWater.jl}: {S}hallow water simulations with {T}rixi.jl}, \bibinfo{howpublished}{\url{https://github.com/trixi-framework/TrixiShallowWater.jl}}, \bibinfo{year}{2025}.
\bibitem[{Ersing et~al.(2025)Ersing, Goldberg, and Winters}]{ersing2025entropy}
\bibinfo{author}{P.~Ersing}, \bibinfo{author}{S.~Goldberg}, \bibinfo{author}{A.~R. Winters},
\newblock \bibinfo{title}{Entropy stable hydrostatic reconstruction schemes for shallow water systems},
\newblock \bibinfo{journal}{Journal of Computational Physics}  (\bibinfo{year}{2025}) \bibinfo{pages}{113802}.
\bibitem[{Fjordholm et~al.(2011)Fjordholm, Mishra, and Tadmor}]{fjordholm2011well}
\bibinfo{author}{U.~S. Fjordholm}, \bibinfo{author}{S.~Mishra}, \bibinfo{author}{E.~Tadmor},
\newblock \bibinfo{title}{Well-balanced and energy stable schemes for the shallow water equations with discontinuous topography},
\newblock \bibinfo{journal}{Journal of Computational Physics} \bibinfo{volume}{230} (\bibinfo{year}{2011}) \bibinfo{pages}{5587--5609}.
\bibitem[{Rueda-Ram{\'\i}rez and Gassner(2024)}]{rueda2024flux}
\bibinfo{author}{A.~M. Rueda-Ram{\'\i}rez}, \bibinfo{author}{G.~J. Gassner},
\newblock \bibinfo{title}{A flux-differencing formula for split-form summation by parts discretizations of non-conservative systems: {A}pplications to subcell limiting for magneto-hydrodynamics},
\newblock \bibinfo{journal}{Journal of Computational Physics} \bibinfo{volume}{496} (\bibinfo{year}{2024}) \bibinfo{pages}{112607}.
\bibitem[{Careaga et~al.(2026)Careaga, Ersing, Koellermeier, and Winters}]{careaga2026entropy}
\bibinfo{author}{J.~Careaga}, \bibinfo{author}{P.~Ersing}, \bibinfo{author}{J.~Koellermeier}, \bibinfo{author}{A.~R. Winters},
\newblock \bibinfo{title}{Entropy analysis and entropy stable {DG} methods for the shallow water moment equations},
\newblock \bibinfo{journal}{arXiv preprint arXiv:2602.06513}  (\bibinfo{year}{2026}).
\bibitem[{Ranocha et~al.(2022)Ranocha, Schlottke-Lakemper, Winters, Faulhaber, Chan, and Gassner}]{ranocha2022adaptive}
\bibinfo{author}{H.~Ranocha}, \bibinfo{author}{M.~Schlottke-Lakemper}, \bibinfo{author}{A.~R. Winters}, \bibinfo{author}{E.~Faulhaber}, \bibinfo{author}{J.~Chan}, \bibinfo{author}{G.~J. Gassner},
\newblock \bibinfo{title}{Adaptive numerical simulations with {T}rixi.jl: {A} case study of {J}ulia for scientific computing},
\newblock \bibinfo{journal}{Proceedings of the JuliaCon Conferences} \bibinfo{volume}{1} (\bibinfo{year}{2022}) \bibinfo{pages}{77}.
\bibitem[{Schlottke-Lakemper et~al.(2021)Schlottke-Lakemper, Winters, Ranocha, and Gassner}]{schlottkelakemper2021purely}
\bibinfo{author}{M.~Schlottke-Lakemper}, \bibinfo{author}{A.~R. Winters}, \bibinfo{author}{H.~Ranocha}, \bibinfo{author}{G.~J. Gassner},
\newblock \bibinfo{title}{A purely hyperbolic discontinuous {G}alerkin approach for self-gravitating gas dynamics},
\newblock \bibinfo{journal}{Journal of Computational Physics} \bibinfo{volume}{442} (\bibinfo{year}{2021}) \bibinfo{pages}{110467}.
\bibitem[{Schlottke-Lakemper et~al.(2025)Schlottke-Lakemper, Gassner, Ranocha, Winters, Chan, and Rueda-Ramírez}]{schlottkelakemper2025trixi}
\bibinfo{author}{M.~Schlottke-Lakemper}, \bibinfo{author}{G.~J. Gassner}, \bibinfo{author}{H.~Ranocha}, \bibinfo{author}{A.~R. Winters}, \bibinfo{author}{J.~Chan}, \bibinfo{author}{A.~Rueda-Ramírez}, \bibinfo{title}{{T}rixi.jl: {A}daptive high-order numerical simulations of hyperbolic {PDE}s in {J}ulia}, \bibinfo{howpublished}{\url{https://github.com/trixi-framework/Trixi.jl}}, \bibinfo{year}{2025}.
\bibitem[{Powell et~al.(1999)Powell, Roe, Linde, Gombosi, and {De Zeeuw}}]{Powell2001}
\bibinfo{author}{K.~G. Powell}, \bibinfo{author}{P.~L. Roe}, \bibinfo{author}{T.~J. Linde}, \bibinfo{author}{T.~I. Gombosi}, \bibinfo{author}{D.~L. {De Zeeuw}},
\newblock \bibinfo{title}{{A Solution-Adaptive Upwind Scheme for Ideal Magnetohydrodynamics}},
\newblock \bibinfo{journal}{Journal of Computational Physics} \bibinfo{volume}{154} (\bibinfo{year}{1999}) \bibinfo{pages}{284--309}.
\bibitem[{Dedner et~al.(2002)Dedner, Kemm, Kr{\"{o}}ner, Munz, Schnitzer, and Wesenberg}]{Dedner2002}
\bibinfo{author}{A.~Dedner}, \bibinfo{author}{F.~Kemm}, \bibinfo{author}{D.~Kr{\"{o}}ner}, \bibinfo{author}{C.~D. Munz}, \bibinfo{author}{T.~Schnitzer}, \bibinfo{author}{M.~Wesenberg},
\newblock \bibinfo{title}{{Hyperbolic divergence cleaning for the MHD equations}},
\newblock \bibinfo{journal}{Journal of Computational Physics} \bibinfo{volume}{175} (\bibinfo{year}{2002}) \bibinfo{pages}{645--673}.
\bibitem[{Derigs et~al.(2017)Derigs, Winters, Gassner, and Walch}]{Derigs2017}
\bibinfo{author}{D.~Derigs}, \bibinfo{author}{A.~R. Winters}, \bibinfo{author}{G.~J. Gassner}, \bibinfo{author}{S.~Walch},
\newblock \bibinfo{title}{{A novel averaging technique for discrete entropy-stable dissipation operators for ideal MHD}},
\newblock \bibinfo{journal}{Journal of Computational Physics} \bibinfo{volume}{330} (\bibinfo{year}{2017}) \bibinfo{pages}{624--632}.
\bibitem[{Baer and Nunziato(1986)}]{baer1986two}
\bibinfo{author}{M.~R. Baer}, \bibinfo{author}{J.~W. Nunziato},
\newblock \bibinfo{title}{A two-phase mixture theory for the deflagration-to-detonation transition ({DDT}) in reactive granular materials},
\newblock \bibinfo{journal}{International journal of multiphase flow} \bibinfo{volume}{12} (\bibinfo{year}{1986}) \bibinfo{pages}{861--889}.
\bibitem[{Chandrashekar and Zenk(2017)}]{chandrashekar2017well}
\bibinfo{author}{P.~Chandrashekar}, \bibinfo{author}{M.~Zenk},
\newblock \bibinfo{title}{Well-balanced nodal discontinuous {G}alerkin method for {E}uler equations with gravity},
\newblock \bibinfo{journal}{Journal of Scientific Computing} \bibinfo{volume}{71} (\bibinfo{year}{2017}) \bibinfo{pages}{1062--1093}.
\bibitem[{Chen et~al.(1996)Chen, Yu, and Wang}]{chen1996reducing}
\bibinfo{author}{W.~Chen}, \bibinfo{author}{Y.~Yu}, \bibinfo{author}{X.~Wang},
\newblock \bibinfo{title}{Reducing the computational requirements of the differential quadrature method},
\newblock \bibinfo{journal}{Numerical Methods for Partial Differential Equations: An International Journal} \bibinfo{volume}{12} (\bibinfo{year}{1996}) \bibinfo{pages}{565--577}.
\bibitem[{Fisher et~al.(2013)Fisher, Carpenter, Nordstr{\"{o}}m, Yamaleev, and Swanson}]{Fisher2013}
\bibinfo{author}{T.~C. Fisher}, \bibinfo{author}{M.~H. Carpenter}, \bibinfo{author}{J.~Nordstr{\"{o}}m}, \bibinfo{author}{N.~K. Yamaleev}, \bibinfo{author}{C.~Swanson},
\newblock \bibinfo{title}{{Discretely conservative finite-difference formulations for nonlinear conservation laws in split form: Theory and boundary conditions}},
\newblock \bibinfo{journal}{Journal of Computational Physics} \bibinfo{volume}{234} (\bibinfo{year}{2013}) \bibinfo{pages}{353--375}.
\bibitem[{Renac(2019)}]{Renac2019}
\bibinfo{author}{F.~Renac},
\newblock \bibinfo{title}{{Entropy stable DGSEM for nonlinear hyperbolic systems in nonconservative form with application to two-phase flows}},
\newblock \bibinfo{journal}{Journal of Computational Physics} \bibinfo{volume}{382} (\bibinfo{year}{2019}) \bibinfo{pages}{1--26}.
\bibitem[{Vol'pert(1967)}]{vol1967spaces}
\bibinfo{author}{A.~I. Vol'pert},
\newblock \bibinfo{title}{The spaces {BV} and quasilinear equations},
\newblock \bibinfo{journal}{Mathematics of the USSR-Sbornik} \bibinfo{volume}{2} (\bibinfo{year}{1967}) \bibinfo{pages}{225}.
\bibitem[{Rueda-Ram{\'{i}}rez and Gassner(2021)}]{Rueda-Ramirez2021}
\bibinfo{author}{A.~M. Rueda-Ram{\'{i}}rez}, \bibinfo{author}{G.~J. Gassner},
\newblock \bibinfo{title}{{A Subcell Finite Volume Positivity-Preserving Limiter for DGSEM Discretizations of the Euler Equations}},
\newblock in: \bibinfo{booktitle}{WCCM-ECCOMAS2020}, pp. \bibinfo{pages}{1--12}.
\bibitem[{Shu and Osher(1988)}]{shu1988efficient}
\bibinfo{author}{C.-W. Shu}, \bibinfo{author}{S.~Osher},
\newblock \bibinfo{title}{Efficient implementation of essentially non-oscillatory shock-capturing schemes},
\newblock \bibinfo{journal}{Journal of computational physics} \bibinfo{volume}{77} (\bibinfo{year}{1988}) \bibinfo{pages}{439--471}.
\bibitem[{Geuzaine and Remacle(2009)}]{geuzaine2009gmsh}
\bibinfo{author}{C.~Geuzaine}, \bibinfo{author}{J.-F. Remacle},
\newblock \bibinfo{title}{Gmsh: {A} {3-D} finite element mesh generator with built-in pre-and post-processing facilities},
\newblock \bibinfo{journal}{International journal for numerical methods in engineering} \bibinfo{volume}{79} (\bibinfo{year}{2009}) \bibinfo{pages}{1309--1331}.
\bibitem[{Kopriva et~al.(2024{\natexlab{a}})Kopriva, Winters, Schlottke-Lakemper, Schoonover, and Ranocha}]{kopriva2024hohqmesh:joss}
\bibinfo{author}{D.~A. Kopriva}, \bibinfo{author}{A.~R. Winters}, \bibinfo{author}{M.~Schlottke-Lakemper}, \bibinfo{author}{J.~A. Schoonover}, \bibinfo{author}{H.~Ranocha},
\newblock \bibinfo{title}{{HOHQM}esh: An all quadrilateral/hexahedral unstructured mesh generator for high order elements},
\newblock \bibinfo{journal}{Journal of Open Source Software} \bibinfo{volume}{9} (\bibinfo{year}{2024}{\natexlab{a}}) \bibinfo{pages}{7476}.
\bibitem[{Kopriva et~al.(2024{\natexlab{b}})Kopriva, Winters, Schlottke-Lakemper, and Ranocha}]{kopriva2024hohqmeshjl}
\bibinfo{author}{D.~A. Kopriva}, \bibinfo{author}{A.~R. Winters}, \bibinfo{author}{M.~Schlottke-Lakemper}, \bibinfo{author}{H.~Ranocha}, \bibinfo{title}{{HOHQM}esh.jl: A {J}ulia frontend to the {F}ortran-based {HOHQM}esh mesh generator for high order elements}, \bibinfo{howpublished}{\url{https://github.com/trixi-framework/HOHQMesh.jl}}, \bibinfo{year}{2024}{\natexlab{b}}.
\bibitem[{Ahrens et~al.(2005)Ahrens, Geveci, and Law}]{ahrens2005paraview}
\bibinfo{author}{J.~Ahrens}, \bibinfo{author}{B.~Geveci}, \bibinfo{author}{C.~Law},
\newblock \bibinfo{title}{{ParaView}: {A}n end-user tool for large-data visualization},
\newblock in: \bibinfo{booktitle}{The Visualization Handbook}, \bibinfo{publisher}{Elsevier}, \bibinfo{year}{2005}, pp. \bibinfo{pages}{717--731}.
\bibitem[{Danisch and Krumbiegel(2021)}]{Makie2021}
\bibinfo{author}{S.~Danisch}, \bibinfo{author}{J.~Krumbiegel},
\newblock \bibinfo{title}{{Makie.jl}: {F}lexible high-performance data visualization for {Julia}},
\newblock \bibinfo{journal}{Journal of Open Source Software} \bibinfo{volume}{6} (\bibinfo{year}{2021}) \bibinfo{pages}{3349}.
\bibitem[{Rueda-Ram\'{i}rez et~al.(2026)Rueda-Ram\'{i}rez, Ersing, Winters, and Gassner}]{rueda2026numericalRepro}
\bibinfo{author}{A.~M. Rueda-Ram\'{i}rez}, \bibinfo{author}{P.~Ersing}, \bibinfo{author}{A.~R. Winters}, \bibinfo{author}{G.~J. Gassner}, \bibinfo{title}{Reproducibility repository for "{W}ell-balanced subcell limiting for discontinuous {G}alerkin discretizations of the shallow-water equations"}, \bibinfo{howpublished}{\url{https://doi.org/10.5281/zenodo.19913123}}, \bibinfo{year}{2026}.
\bibitem[{Zalesak(1979)}]{zalesak1979fully}
\bibinfo{author}{S.~T. Zalesak},
\newblock \bibinfo{title}{Fully multidimensional flux-corrected transport algorithms for fluids},
\newblock \bibinfo{journal}{Journal of Computational Physics} \bibinfo{volume}{31} (\bibinfo{year}{1979}) \bibinfo{pages}{335--362}.
\bibitem[{Kuzmin et~al.(2010)Kuzmin, M{\"o}ller, Shadid, and Shashkov}]{kuzmin2010failsafe}
\bibinfo{author}{D.~Kuzmin}, \bibinfo{author}{M.~M{\"o}ller}, \bibinfo{author}{J.~N. Shadid}, \bibinfo{author}{M.~Shashkov},
\newblock \bibinfo{title}{Failsafe flux limiting and constrained data projections for equations of gas dynamics},
\newblock \bibinfo{journal}{Journal of Computational physics} \bibinfo{volume}{229} (\bibinfo{year}{2010}) \bibinfo{pages}{8766--8779}.
\bibitem[{Kuzmin et~al.(2012)Kuzmin, Löhner, and Turek}]{kuzmin2012}
\bibinfo{editor}{D.~Kuzmin}, \bibinfo{editor}{R.~Löhner}, \bibinfo{editor}{S.~Turek} (Eds.), \bibinfo{title}{Flux-Corrected Transport: {P}rinciples, Algorithms, and Applications}, \bibinfo{publisher}{Springer}, \bibinfo{address}{Dordrecht}, \bibinfo{edition}{2} edition, \bibinfo{year}{2012}.
\bibitem[{Lohmann et~al.(2017)Lohmann, Kuzmin, Shadid, and Mabuza}]{lohmann2017}
\bibinfo{author}{C.~Lohmann}, \bibinfo{author}{D.~Kuzmin}, \bibinfo{author}{J.~N. Shadid}, \bibinfo{author}{S.~Mabuza},
\newblock \bibinfo{title}{Flux-corrected transport algorithms for continuous {G}alerkin methods based on high order {B}ernstein finite elements},
\newblock \bibinfo{journal}{Journal of Computational Physics} \bibinfo{volume}{344} (\bibinfo{year}{2017}) \bibinfo{pages}{151--186}.
\bibitem[{Persson and Peraire(2006)}]{Persson2006}
\bibinfo{author}{P.-O. Persson}, \bibinfo{author}{J.~Peraire},
\newblock \bibinfo{title}{{Sub-Cell Shock Capturing for Discontinuous Galerkin Methods}},
\newblock \bibinfo{journal}{44th AIAA Aerospace Sciences Meeting and Exhibit}  (\bibinfo{year}{2006}) \bibinfo{pages}{1--13}.
\bibitem[{Soares-Fraz{\~a}o and Zech(2007)}]{soares2007experimental}
\bibinfo{author}{S.~Soares-Fraz{\~a}o}, \bibinfo{author}{Y.~Zech},
\newblock \bibinfo{title}{Experimental study of dam-break flow against an isolated obstacle},
\newblock \bibinfo{journal}{Journal of Hydraulic Research} \bibinfo{volume}{45} (\bibinfo{year}{2007}) \bibinfo{pages}{27--36}.
\bibitem[{Chertock et~al.(2015)Chertock, Cui, Kurganov, and Wu}]{chertock2015well}
\bibinfo{author}{A.~Chertock}, \bibinfo{author}{S.~Cui}, \bibinfo{author}{A.~Kurganov}, \bibinfo{author}{T.~Wu},
\newblock \bibinfo{title}{Well-balanced positivity preserving central-upwind scheme for the shallow water system with friction terms},
\newblock \bibinfo{journal}{International Journal for numerical methods in fluids} \bibinfo{volume}{78} (\bibinfo{year}{2015}) \bibinfo{pages}{355--383}.
\bibitem[{Bonev et~al.(2018)Bonev, Hesthaven, Giraldo, and Kopera}]{bonev2018discontinuous}
\bibinfo{author}{B.~Bonev}, \bibinfo{author}{J.~S. Hesthaven}, \bibinfo{author}{F.~X. Giraldo}, \bibinfo{author}{M.~A. Kopera},
\newblock \bibinfo{title}{Discontinuous {G}alerkin scheme for the spherical shallow water equations with applications to tsunami modeling and prediction},
\newblock \bibinfo{journal}{Journal of Computational Physics} \bibinfo{volume}{362} (\bibinfo{year}{2018}) \bibinfo{pages}{425--448}.
\bibitem[{Ginting(2019)}]{ginting2019central}
\bibinfo{author}{B.~M. Ginting},
\newblock \bibinfo{title}{Central-upwind scheme for {2D} turbulent shallow flows using high-resolution meshes with scalable wall functions},
\newblock \bibinfo{journal}{Computers \& Fluids} \bibinfo{volume}{179} (\bibinfo{year}{2019}) \bibinfo{pages}{394--421}.
\bibitem[{Cea and Blad{\'e}(2015)}]{cea2015simple}
\bibinfo{author}{L.~Cea}, \bibinfo{author}{E.~Blad{\'e}},
\newblock \bibinfo{title}{A simple and efficient unstructured finite volume scheme for solving the shallow water equations in overland flow applications},
\newblock \bibinfo{journal}{Water resources research} \bibinfo{volume}{51} (\bibinfo{year}{2015}) \bibinfo{pages}{5464--5486}.
\bibitem[{Hou et~al.(2013)Hou, Liang, Simons, and Hinkelmann}]{hou20132d}
\bibinfo{author}{J.~Hou}, \bibinfo{author}{Q.~Liang}, \bibinfo{author}{F.~Simons}, \bibinfo{author}{R.~Hinkelmann},
\newblock \bibinfo{title}{A 2{D} well-balanced shallow flow model for unstructured grids with novel slope source term treatment},
\newblock \bibinfo{journal}{Advances in Water Resources} \bibinfo{volume}{52} (\bibinfo{year}{2013}) \bibinfo{pages}{107--131}.
\bibitem[{Ayog et~al.(2021)Ayog, Kesserwani, Shaw, Sharifian, and Bau}]{ayog2021second}
\bibinfo{author}{J.~L. Ayog}, \bibinfo{author}{G.~Kesserwani}, \bibinfo{author}{J.~Shaw}, \bibinfo{author}{M.~K. Sharifian}, \bibinfo{author}{D.~Bau},
\newblock \bibinfo{title}{Second-order discontinuous {G}alerkin flood model: {C}omparison with industry-standard finite volume models},
\newblock \bibinfo{journal}{Journal of Hydrology} \bibinfo{volume}{594} (\bibinfo{year}{2021}) \bibinfo{pages}{125924}.
\bibitem[{Kuzmin et~al.(2022)Kuzmin, Hajduk, and Rupp}]{kuzmin2022limiter}
\bibinfo{author}{D.~Kuzmin}, \bibinfo{author}{H.~Hajduk}, \bibinfo{author}{A.~Rupp},
\newblock \bibinfo{title}{Limiter-based entropy stabilization of semi-discrete and fully discrete schemes for nonlinear hyperbolic problems},
\newblock \bibinfo{journal}{Computer Methods in Applied Mechanics and Engineering} \bibinfo{volume}{389} (\bibinfo{year}{2022}) \bibinfo{pages}{114428}.
\bibitem[{Christner and Chan(2025)}]{christner2025entropy}
\bibinfo{author}{B.~Christner}, \bibinfo{author}{J.~Chan},
\newblock \bibinfo{title}{Entropy stable finite difference methods via entropy correction artificial viscosity and knapsack limiting},
\newblock \bibinfo{journal}{arXiv preprint arXiv:2508.21226}  (\bibinfo{year}{2025}).

\end{thebibliography}

\section*{Appendices}
\appendix

We describe a generic 2D version of the node-wise subcell limiting scheme in Appendix~\ref{app:2d}. Specific details regarding the 2D shallow water equations are provided in Appendix~\ref{app:swe_2d}. Finally, we describe the entropic properties of the node-wise subcell limiting formulation in Appendix~\ref{app:entProps}.

\section{Extension to 2D discretizations and unstructured curvilinear meshes}\label{app:2d}

The extension to two-dimensional discretizations and unstructured curvilinear meshes follows analogously to \cite[Appendix B]{rueda2024flux}. 
A two-dimensional extension of system \eqref{eq:noncons-system} is given by
\begin{equation} \label{eq:noncons-system_2D}
\partial_t \mathbf{u} 
+ \Nabla \cdot \blocktensor{f} (\mathbf{u}) 
+ \noncon(\mathbf{u}, \Nabla \mathbf{u})
= \state{0},
\end{equation}
where we introduced the block vector notation $\blocktensor{f} = \left(\state{f}_1, \state{f}_2\right)^T$ for fluxes in the $x$ and $y$ directions.

Using a tensor-product construction of \eqref{eq:DGSEM}, we obtain a high-order DGSEM discretization on two-dimensional curvilinear meshes. 
The resulting scheme reads \cite{rueda2023entropy}
\begin{align}\label{eq:dgsem_2d}
J_{ij} \omega_{ij} \dot{\state{u}}^{\DG}_{ij} 
&+ 
\omega_{j} \left[
  \sum_{m=0}^N S_{im} 
    \left(
    \tilde{\state{f}}^{1*}_{(i,m)j}
    +\numnonconsSxi{\tilde{\Jan}}_{(i,m)j}
    \right)
 + \delta_{iN} \left(
  \numfluxb{f}_{(N,R)j}
    +\numnonconsD{\Jan}_{(N,R)j}
  \right)
 - \delta_{i0} \left(
  \numfluxb{f}_{(0,L)j}
    +\numnonconsD{\Jan}_{(0,L)j}
  \right)
\right]
\nonumber\\
&+ 
\omega_{i} \left[
  \sum_{m=0}^N S_{jm}
    \left(
    \tilde{\state{f}}^{2*}_{i(j,m)}
    +\numnonconsSeta{\tilde{\Jan}}_{i(j,m)}
    \right)
 + \delta_{jN} \left(
  \numfluxb{f}_{i(N,R)}
    +\numnonconsD{\Jan}_{i(N,R)}
  \right)
 - \delta_{j0} \left(
  \numfluxb{f}_{i(0,L)}
    +\numnonconsD{\Jan}_{i(0,L)}
  \right)
\right]= \state{0}.
\end{align}

Here, $\tilde{\state{f}}^{k*}$ denotes the two-point contravariant numerical volume flux in the $k$-th direction, computed as the product of a two-point flux and the averaged contravariant metric terms ($\vec{a}^k \approx \Nabla \xi^k$),
\begin{align}
\tilde{\state{f}}^{1*}_{(i,m)j} &:= \blocktensor{f}^{*}(\state{u}_{ij}, \state{u}_{mj}) \cdot \avg{J\vec{a}^1}_{(i,m)j}, ~~~~~
\tilde{\state{f}}^{2*}_{i(j,m)} := \blocktensor{f}^{*}(\state{u}_{ij}, \state{u}_{im}) \cdot \avg{J\vec{a}^2}_{i(j,m)},
\end{align}
while $\numfluxb{f}$ is the corresponding surface contribution
\begin{equation}
    \begin{aligned}
    \hat{\state{f}}_{(0,L)j} &:= \hat{\blocktensor{\state{f}}}(\state{u}_{0j}, \state{u}_{Lj}) \cdot \left(J\vec{a}^1\right)_{0j},
    \quad
    \hat{\state{f}}_{(N,R)j} := \hat{\blocktensor{\state{f}}} (\state{u}_{Nj}, \state{u}_{Rj}) \cdot \left(J\vec{a}^1\right)_{Nj},
    \\
    \hat{\state{f}}_{i(0,L)} &:= \hat{\blocktensor{\state{f}}} (\state{u}_{i0}, \state{u}_{iL}) \cdot \left(J\vec{a}^2\right)_{i0},
    \quad
    \hat{\state{f}}_{i(N,R)} := \hat{\blocktensor{\state{f}}} (\state{u}_{iN}, \state{u}_{iR}) \cdot \left(J\vec{a}^2\right)_{iN}.
    \end{aligned}
\end{equation}

Analogously, the quantities $\tilde{\Jan}^{k\star}$ denote two-point contravariant non-conservative volume terms, while $\numnonconsD{\Jan}$ is the non-conservative surface contribution.
Following the one-dimensional analysis (see Proposition \ref{prop:fluxdiff}), if the non-conservative terms admit a factorization into local and jump contributions, with non-conservative volume terms
\begin{equation}
\begin{aligned}
\tilde{\Jan}^{1\star}_{(i,m)j} :=\ & \tilde{\Jan}^{1\star} \left(\state{u}_{ij}, \state{u}_{mj}, (J\vec{a}^1)_{ij}, (J\vec{a}^1)_{mj}\right) = \tilde{\Jan}^{1,\mathrm{loc}}_{ij} \circ \tilde{\Jan}^{1,\mathrm{jump}}_{(i,m)j},
\\
\tilde{\Jan}^{2\star}_{i(j,m)} :=\ & \tilde{\Jan}^{2\star} \left(\state{u}_{ij}, \state{u}_{im}, (J\vec{a}^2)_{ij}, (J\vec{a}^2)_{im}\right) = \tilde{\Jan}^{2,\mathrm{loc}}_{ij} \circ \tilde{\Jan}^{2,\mathrm{jump}}_{i(j,m)},
\end{aligned}
\end{equation}
and corresponding surface terms
\begin{equation}
\begin{aligned}
\numnonconsDxi{\Jan}_{(0,L)j} :=\ &\tilde{\Jan}^{1\star} \left(\state{u}_{0j}, \state{u}_{Lj}, (J\vec{a}^1)_{0j}, (J\vec{a}^1)_{0j}\right),
\quad
\numnonconsDxi{\Jan}_{(N,R)j} :=\ \tilde{\Jan}^{1\star} \left(\state{u}_{Nj}, \state{u}_{Rj}, (J\vec{a}^1)_{Nj}, (J\vec{a}^1)_{Nj}\right),\\
\numnonconsDeta{\Jan}_{i(0,L)} :=\ &\tilde{\Jan}^{2\star} \left(\state{u}_{i0}, \state{u}_{iL}, (J\vec{a}^2)_{i0}, (J\vec{a}^2)_{i0}\right),
\quad
\numnonconsDeta{\Jan}_{i(N,R)} :=\ \tilde{\Jan}^{2\star} \left(\state{u}_{iN}, \state{u}_{iR}, (J\vec{a}^2)_{iN}, (J\vec{a}^2)_{iN}\right),
\end{aligned}
\end{equation}
then \eqref{eq:dgsem_2d} can be rewritten in flux-differencing form,
\begin{equation} \label{eq:fluxdiff2d}
    J_{ij} \dot{\state{u}}^{\DG}_{ij} = 
        \frac{1}{\omega_i}
        \left(
        \stateG{\Gamma}^{1,\DG}_{(i,i-1)j}
        -\stateG{\Gamma}^{1,\DG}_{(i,i+1)j}
        \right)
        +
        \frac{1}{\omega_j}
        \left(
        \stateG{\Gamma}^{2,\DG}_{i(j,j-1)}
        -\stateG{\Gamma}^{2,\DG}_{i(j,j+1)}
        \right)
    ,
    \qquad
    i,j=0, \ldots, N.
\end{equation}
The staggered flux terms have the following explicit expressions:
\begin{align}
\stateG{\Gamma}^{1,\DG}_{(0,-1)j}  =& \  \numfluxb{f}_{(0,L)j} + \numnonconsD{\tilde{\Jan}}_{(0,L)j}, 
\label{eq:leftFlux_xi}
\\
\stateG{\Gamma}^{1,\DG}_{(i,k)j} =& \   \sum_{l=0}^{\min(i,k)} \sum_{m=0}^N S_{lm} \tilde{\state{f}}^{1*}_{(l,m)j}+\tilde{\Jan}^{1,\sloc}_{ij} \circ \sum_{l=0}^{\min(i,k)} \sum_{m=0}^N S_{lm} \tilde{\Jan}^{1,\sjump}_{(l,m)j} 
+ 2 \tilde{\Jan}^{1,\sloc}_{ij} \circ \tilde{\Jan}^{1,\sjump}_{(i,0)j}, 
& \nonumber\\
&\forall (i, k) = (i,i \pm 1) \setminus \{(0,-1), (N,N+1) \},
\label{eq:internFlux_xi}\\
\stateG{\Gamma}^{1,\DG}_{(N,N+1)j} =& \  \numfluxb{f}_{(N,R)j} +  \numnonconsD{\tilde{\Jan}}_{(N,R)j},
\label{eq:rightFlux_xi}
\\
\stateG{\Gamma}^{2,\DG}_{i(0,-1)}  =& \  \numfluxb{f}_{i(0,L)} + \numnonconsD{\tilde{\Jan}}_{i(0,L)}, 
\label{eq:leftFlux_eta}
\\
\stateG{\Gamma}^{2,\DG}_{i(j,k)} =& \   \sum_{l=0}^{\min(j,k)} \sum_{m=0}^N S_{lm} \tilde{\state{f}}^{2*}_{i(l,m)}+\tilde{\Jan}^{2,\sloc}_{ij} \circ \sum_{l=0}^{\min(j,k)} \sum_{m=0}^N S_{lm} \tilde{\Jan}^{2,\sjump}_{i(l,m)} 
+ 2 \tilde{\Jan}^{2,\sloc}_{ij} \circ \tilde{\Jan}^{2,\sjump}_{i(j,0)}, 
& \nonumber\\
&\forall (j, k) = (j,j \pm 1) \setminus \{(0,-1), (N,N+1) \},
\label{eq:internFlux_eta}\\
\stateG{\Gamma}^{2,\DG}_{i(N,N+1)} =& \  \numfluxb{f}_{i(N,R)} +  \numnonconsD{\tilde{\Jan}}_{i(N,R)}.
\label{eq:rightFlux_eta}
\end{align}

\section{Well-balanced Formulation for the 2D-Shallow Water Equations}\label{app:swe_2d}

As discussed in Corollary~\ref{cor:nodewise_balance}, the novel flux-differencing formula in local-jump formulation introduced in Section~\ref{sec:fluxdiff} is high-order and equilibrium preserving, provided that fluxes and non-conservative terms vanish node-wise at the equilibrium state.
In the one-dimensional setting this can be used to construct an equilibrium preserving scheme for the shallow water equations if the two-point fluxes \eqref{eq:ersing_fluxes} are used.

In the following, we show how this formulation can be adapted for the 2D discretization on unstructured curvilinear meshes described in Appendix~{\ref{app:2d}}. We consider the following formulation of the two-dimensional shallow-water equations, in which the pressure contribution is separated from the flux and incorporated into the non-conservative term:
\begin{equation}\label{eq:swe2d}
    \partial_t
	\begin{pmatrix}
		h \\ h v_1 \\ h v_2 
	\end{pmatrix}
	+
    \partial_x
    \underbrace{
	\begin{pmatrix}
		hv_1 \\
		hv_1^2\\
		hv_1v_2\\
	\end{pmatrix}}_{\state{f}^1}
	+
    \partial_y
    \underbrace{
	\begin{pmatrix}
		hv_2 \\
		hv_1v_2 \\
		hv_2^2\\
	\end{pmatrix}}_{\state{f}^2}
	+
    \underbrace{
    \begin{pmatrix}
		0 \\ 
		gh \partial_x (b+h)\\
		gh \partial_y (b+h)\\
	\end{pmatrix}}_{\stateG{\Upsilon}}
    = \state{0}
\end{equation}

Analogously to the one-dimensional setting, we use the two-point fluxes of \citet{ersing2025entropy} for the volume numerical fluxes:
\begin{align}\label{eq:ersing_fluxes_2d_again}
   \blocktensor{f}^{*}(\state{u}_{L}, \state{u}_{R})  =
   \begin{pmatrix}
       \state{f}^{1*}(\state{u}_{L}, \state{u}_{R})
       \\[0.1cm]
       \state{f}^{2*}(\state{u}_{L}, \state{u}_{R})
   \end{pmatrix} =
   \left(
    \begin{pmatrix}
        \avg{h v_1} \\[0.1cm]
        \avg{h v_1} \avg{v_1}\\[0.1cm]
        \avg{h v_1} \avg{v_2}
    \end{pmatrix},
    \begin{pmatrix}
        \avg{h v_2} \\[0.1cm]
        \avg{h v_1} \avg{v_2}\\[0.1cm]
        \avg{h v_2} \avg{v_2}
    \end{pmatrix}
    \right)^T,
\end{align}
and the following modification of the contravariant non-conservative two-point terms by \citet{ersing2025entropy}:
\begin{align}\label{eq:ersing_fluxes_noncons_2d}
    \numnonconsSxi{\tilde{\Jan}}_{(i,m)j} = 
     \underbrace{
     \begin{pmatrix}
        0 \\  \frac{g}{2}h_{ij}(Ja^1_1)_{ij} \\[0.1cm]  \frac{g}{2}h_{ij}(Ja^1_2)_{ij} 
    \end{pmatrix}
    }
    _{\tilde{\Jan}^{1,\sloc}_{ij}}
    \circ
    \underbrace{
    \begin{pmatrix}
        0 \\
        \jump{h + b}_{(i,m)j}\\
        \jump{h + b}_{(i,m)j}
    \end{pmatrix}
    }_{\tilde{\Jan}^{1,\sjump}_{(i,m)j}},
    \quad
    \numnonconsSeta{\tilde{\Jan}}_{i(j,m)} = 
    \underbrace{
     \begin{pmatrix}
        0 \\  \frac{g}{2}h_{ij}(Ja^2_1)_{ij} \\[0.1cm]  \frac{g}{2}h_{ij}(Ja^2_2)_{ij} 
    \end{pmatrix}
    }_{\tilde{\Jan}^{2,\sloc}_{ij}}
    \circ
    \underbrace{
    \begin{pmatrix}
    0 \\
    \jump{h + b}_{i(j,m)} \\
    \jump{h + b}_{i(j,m)}
    \end{pmatrix}
    }_{\tilde{\Jan}^{2,\sjump}_{i(j,m)}}.
\end{align}

Note that the proposed non-conservative terms are not algebraically identical to those in \citet{ersing2025entropy}, since the metric terms are incorporated as a local term rather than an average.
This modification is necessary to obtain a factorization into local and jump contributions. 
In general, the metric term can also be incorporated into the jump term. However, on curvilinear meshes this would not yield a well-balanced scheme, as the jump term does not vanish for the lake-at-rest steady state due to the spatially varying metric terms.

For the numerical interface fluxes, we use the same formulation as for the numerical volume fluxes \eqref{eq:ersing_fluxes_2d_again} together with a modified local Lax-Friedrichs dissipation term written in entropy variables that is well-balanced for discontinuous bottom topography.
The resulting numerical interface fluxes in direction $k$ is given by
\begin{align}
    \numfluxb{f}_{(i,k)j} =& \  \hat{\blocktensor{\state{f}}}(\state{u}_{ij}, \state{u}_{kj}) \cdot \left(J\vec{a}^1\right)_{ij} 
    - \frac{\sigma_{iN}}{2} \big|\lambda^{\max}_{(i,k)j}\big| \,\norm{(J\vec{a}^1)_{ij}}_2 \state{H}_{(i,k)j}\jump{\state{v}}_{(i,k)j},
    \\
    \numfluxb{f}_{i(j,k)} =& \ \hat{\blocktensor{\state{f}}}(\state{u}_{ij}, \state{u}_{ik}) \cdot \left(J\vec{a}^2\right)_{ij} 
    - \frac{\sigma_{jN}}{2} \big|\lambda^{\max}_{i(j,k)}\big| \, \norm{(J\vec{a}^2)_{ij}}_2 \state{H}_{i(j,k)}\jump{\state{v}}_{i(j,k)},
\end{align}
where $\lambda^{\max}_{(\cdot,\cdot)}$ denotes the maximum eigenvalue in direction between two states, $\state{v}(\state{u}) := (g(h+b)-0.5\|\vec{v}\|_2^2,\ v_1,\ v_2 )^T$ are the entropy variables, $\sigma_{ij}:=2\delta_{ij}-1$ is an auxiliary variable that ensures that the dissipation is applied in a conservative manner, and
\begin{equation}
    \state{H}_{(L,R)} := \state{H}(\state{u}_L, \state{u}_R) = \frac{1}{g}\begin{pmatrix}
        1 & \avg{v_1} & \avg{v_2} \\
        \avg{v_1} & g\avg{h} + \avg{v_1}^2 & \avg{v_1}\avg{v_2} \\
        \avg{v_2} & \avg{v_1}\avg{v_2} & g\avg{h} + \avg{v_2}^2
    \end{pmatrix},
\end{equation}
is a symmetric positive definite matrix evaluated between two states, which is a discrete version of the inverse Hessian of the total energy function for the shallow water equations.
The matrix $\state{H}$ accounts for the variable change $\jump{\state{u}} \simeq \state{H}{\jump{\state{v}}}$, where the relation is exact for constant bottom topography \cite{ranocha2017shallow}.

\section{Entropy Analysis}\label{app:entProps}
To analyze the entropic properties of the well-balanced, node-wise flux-differencing formulation it is convenient to rewrite the contravariant non-conservative two-point term. We factor the contravariant local term into non-conservative contributions and contravariant metric terms and use that the jump contribution is dimension-independent such that it can be written as a scalar to obtain
\begin{equation}
    \tilde{\Jan}^{1\star}_{(i,m)j} := \left(\blocktensor{\Jan}^{1,\sloc}_{ij} \cdot (J\vec{a}^1)_{ij}\right) \Phi^{1,\sjump}_{(i,m)j}, \quad
    \tilde{\Jan}^{2\star}_{i(j,m)} := \left(\blocktensor{\Jan}^{2,\sloc}_{ij} \cdot (J\vec{a}^2)_{ij}\right) \Phi^{2,\sjump}_{i(j,m)}, 
\end{equation}
where 
\begin{equation}
    \blocktensor{\Jan}^{1,\sloc}_{ij} = \blocktensor{\Jan}^{2,\sloc}_{ij} =
    \left(\begin{pmatrix} 0 \\ \frac{1}{2} g h_{ij} \\ 0 \end{pmatrix},
    \begin{pmatrix} 0 \\ 0 \\ \frac{1}{2} g h_{ij} \end{pmatrix}\right)^T,
    \quad
    \Phi^{1,\sjump}_{(\cdot, \cdot)} = \Phi^{2,\sjump}_{(\cdot, \cdot)} = \jump{h + b}_{(\cdot, \cdot)}. 
\end{equation}
The original fluxes and non-conservative terms of \citet{ersing2025entropy}, incorporating the averaged contravariant metric terms, are entropy conservative, i.e., they satisfy a generalized Tadmor condition:
\begin{equation}\label{eq:ent_prod_ersing}
    r_{(i,m)j} = \jump{\entVar}_{(i,m)j}^T 
    \tilde{\state{f}}^{1*}_{(i,m)j}
    +
    \entVar_{mj}^T \tilde{\Jan}^{1\star E}_{(m,i)j}
    -
    \entVar_{ij}^T \tilde{\Jan}^{1\star E}_{(i,m)j}
    - \avg{J\vec{a}^1}_{(i,m)jk}\cdot \jump{\blocktensor{\psi}}_{(i,m)jk} = 0,
\end{equation}
where $\entVar = \left(g(h+b)-0.5\|\vec{v}\|^2, v_1, v_2\right)^T$ are the entropy variables \cite{wintermeyer2017entropy, ranocha2017shallow, ersing2025entropy}.


To assess entropy conservation for the modified fluxes \eqref{eq:ersing_fluxes_2d_again}, \eqref{eq:ersing_fluxes_noncons_2d}, we analyze the corresponding entropy production terms. After many algebraic manipulation, the most important of which highlighted in orange below, we obtain
\begin{align*}
    r_{(i,m)j} =& \jump{\entVar}_{(i,m)j}^T 
    \tilde{\state{f}}^{1*}_{(i,m)j}
    +
    \entVar_{mj}^T \numnonconsSxi{\tilde{\Jan}}_{(m,i)j}
    -
    \entVar_{ij}^T \numnonconsSxi{\tilde{\Jan}}_{(i,m)j}
    - \avg{J\vec{a}^1}_{(i,m)jk}\cdot \jump{\blocktensor{\psi}}_{(i,m)jk}
    \\
    =& \jump{\entVar}_{(i,m)j}^T 
    \blocktensor{f}^{*}(\state{u}_{ij}, \state{u}_{mj}) \cdot \avg{J\vec{a}^1}_{(i,m)j}
    - \avg{J\vec{a}^1}_{(i,m)j}\cdot \jump{\blocktensor{\psi}}_{(i,m)j}\nonumber
    \\& \;
    +\entVar_{mj}^T \left(\blocktensor{\Jan}^{1,\sloc}_{mj} \cdot (J\vec{a}^1)_{mj}\right) \Phi^{1,\sjump}_{(m,i)j}
    -
    \entVar_{ij}^T \left(\blocktensor{\Jan}^{1,\sloc}_{ij} \cdot (J\vec{a}^1)_{ij}\right) \Phi^{1,\sjump}_{(i,m)j}
    \\
    =&
    \jump{\entVar}_{(i,m)j}^T 
    \blocktensor{f}^{*}(\state{u}_{ij}, \state{u}_{mj}) \cdot \avg{J\vec{a}^1}_{(i,m)j}
    - \avg{J\vec{a}^1}_{(i,m)j}\cdot \jump{\blocktensor{\psi}}_{(i,m)j}\nonumber
    \\& \;
    \change{-}\,\entVar_{mj}^T \left(\blocktensor{\Jan}^{1,\sloc}_{mj} \cdot (J\vec{a}^1)_{mj}\right) \change{\Phi^{1,\sjump}_{(i,m)j}}
    -
    \entVar_{ij}^T \left(\blocktensor{\Jan}^{1,\sloc}_{ij} \cdot (J\vec{a}^1)_{ij}\right) \Phi^{1,\sjump}_{(i,m)j}
    \\
    =&
    \jump{\entVar}_{(i,m)j}^T 
    \blocktensor{f}^{*}(\state{u}_{ij}, \state{u}_{mj}) \cdot \avg{J\vec{a}^1}_{(i,m)j}
    - \avg{J\vec{a}^1}_{(i,m)j}\cdot \jump{\blocktensor{\psi}}_{(i,m)j} \nonumber
    \\& \;
    \change{-}
    \change{2\avg{\left(\entVar^T \blocktensor{\Jan}^{1,\sloc}\right)\cdot (J\vec{a}^1)}_{(i,m)j}} \Phi^{1,\sjump}_{(i,m)j}
\end{align*}

Using the identity $\avg{ab} = \avg{a}\avg{b} + \frac{1}{4}\jump{a}\jump{b},$
we collect all contributions involving averaged contravariant metric terms to obtain the entropy residual
\begin{align}
    r_{(i,m)j} =& \jump{\entVar}_{(i,m)j}^T 
    \blocktensor{f}^{*}(\state{u}_{ij}, \state{u}_{mj}) \cdot \avg{J\vec{a}^1}_{(i,m)j}
    - \avg{J\vec{a}^1}_{(i,m)j}\cdot \jump{\blocktensor{\psi}}_{(i,m)j}\nonumber
    \\&
    \change{-2\avg{\left(\entVar^T \blocktensor{\Jan}^{1,\sloc}\right)}_{(i,m)j} \cdot \avg{J\vec{a}^1}_{(i,m)j}} \Phi^{1,\sjump}_{(i,m)j}
    \change{-\frac{1}{2}\left(\jump{\left(\entVar^T \blocktensor{\Jan}^{1,\sloc}\right)}_{(i,m)j} \cdot \jump{J\vec{a}^1}_{(i,m)j}\right)} \Phi^{1,\sjump}_{(i,m)j}
    \\=&
    \left(\jump{\entVar}_{(i,m)j}^T 
    \blocktensor{f}^{*}(\state{u}_{ij}, \state{u}_{mj}) 
    -\jump{\blocktensor{\psi}}_{(i,m)j}
    - 2\avg{\left(\entVar^T \blocktensor{\Jan}^{1,\sloc}\right)}_{(i,m)j}\Phi^{1,\sjump}_{(i,m)j}\right)\cdot \avg{J\vec{a}^1}_{(i,m)j}\nonumber
    \\
    &\; - \frac{1}{2}\left(\jump{\left(\entVar^T \blocktensor{\Jan}^{1,\sloc}\right)}_{(i,m)j} \cdot \jump{J\vec{a}^1}_{(i,m)j}\right) \Phi^{1,\sjump}_{(i,m)j}
    \label{eq:ent_prod_ersing2}
    \\=& \label{eq:jumpJa}
    - \frac{1}{2}\left(\jump{\left(\entVar^T \blocktensor{\Jan}^{1,\sloc}\right)}_{(i,m)j} \cdot \jump{J\vec{a}^1}_{(i,m)j}\right) \Phi^{1,\sjump}_{(i,m)j}.
\end{align}

The terms appearing in the first line of \eqref{eq:ent_prod_ersing2} vanish, since they satisfy the generalized Tadmor condition associated with the original fluxes of \citet{ersing2025entropy}, as shown in \eqref{eq:ent_prod_ersing}. 
However, the remaining contribution in \eqref{eq:jumpJa} does not vanish in general. 
Consequently, the particular reformulation of the non-conservative two-point terms by \citet{ersing2025entropy} that we propose in \eqref{eq:ersing_fluxes_noncons_2d}, while locally well-balanced and compatible with the novel flux-differencing formulation, does not guarantee semi-discrete entropy conservation on general curvilinear meshes.

We further note that node-wise blending coefficients in hybrid DGSEM/FV subcell limiting do not, in general, preserve semi-discrete entropy conservation or entropy stability, even when both the DGSEM and FV discretizations are individually entropy conservative/stable for arbitrary blending coefficients \cite{lin2024high}. This contrasts with element-wise limiting, for which semi-discrete entropy consistency is maintained for any choice of $\alpha$.

Recent work by \citet{lin2024high,kuzmin2022limiter,christner2025entropy} has shown that the loss of semi-discrete entropy consistency arising from nodal limiting can be remedied through an entropy-aware selection of the blending coefficients, provided that the underlying low-order FV scheme is entropy stable. In particular, the limiting coefficients $\alpha$ can be chosen to enforce semi-discrete entropy consistency. The same strategy can be applied to compensate for the entropy production introduced by the modified two-point terms of \citet{ersing2025entropy} given in \eqref{eq:ersing_fluxes_2d_again}, \eqref{eq:ersing_fluxes_noncons_2d}.

\begin{remark}
    The well-balanced, subcell limiting strategy is entropy conservative provided the jump in the metric terms
    \[
    \jump{J\vec{a}^1}_{(i,m)j} = 0.
    \]
    This holds on straight-sided quadrilateral meshes provided that all elements in the mesh are parallelograms, e.g., Cartesian boxes, rectangular elements, or diamond elements.
\end{remark}

\end{document}